\documentclass[journal,twoside,web]{ieeecolor}

\usepackage{generic}
\usepackage{amsmath,amssymb,amsfonts}
\usepackage{dashbox}
\usepackage{algorithm}
\usepackage{graphicx}
\usepackage{textcomp}
\usepackage{subfigure}
\usepackage{epsfig} 
\usepackage{stfloats}
\usepackage{lcsys}
\usepackage{subcaption}
\usepackage{multirow}%
\usepackage{mathrsfs}%
\usepackage{xcolor}%
\usepackage{manyfoot}%
\usepackage{booktabs}%
\usepackage{algpseudocode}%
\usepackage{listings}
\usepackage{nicefrac}
\usepackage[normalem]{ulem}
\usepackage{multirow, lipsum}
\usepackage{enumerate}
\usepackage{physics}
\usepackage{url}
\usepackage{import}
\usepackage{placeins}
\usepackage{empheq}
\usepackage{float}
\usepackage{multicol}
\usepackage[english]{babel}
\usepackage{pifont}
\usepackage{footnote}
\usepackage{array}
\usepackage[colorlinks=true,linkcolor=blue,citecolor=blue,urlcolor=blue]{hyperref}
\usepackage{cite}
\usepackage{cleveref}
\usepackage{bm}
\usepackage{booktabs}

\newtheorem{theorem}{Theorem}
\newtheorem{definition}{Definition}

\newtheorem{corollary}[theorem]{Corollary}

\newcommand{\hide}[1]{}
\newcommand{\coo}{\ensuremath{\mathrm{CO_2} } }
\newcommand{\N}{\mathcal{N}}
\newcommand{\m}{\mathfrak{m}}
\newcommand{\NEc}{\mathcal{N}_{1}}
\newcommand{\NE}{\mathcal{N}_{\balpha}}

\newcommand{\support}{\mathcal{S}}
\newcommand{\eop}{\hfill \textcolor{black}{\rule{1.5ex}{1.5ex}}}

\newcommand{\bmu}{{\bm \mu}}
\newcommand{\balpha}{{\bm \alpha}}

\newcommand{\A}{{\mathcal A}}
\renewcommand{\t}{\theta}

\newcommand{\balp}{{\bm \alpha}}

\begin{document}
\def\BibTeX{{\rm B\kern-.05em{\sc i\kern-.025em b}\kern-.08em
    T\kern-.1667em\lower.7ex\hbox{E}\kern-.125emX}}
\markboth{\journalname, VOL. XX, NO. XX, XXXX 2017}
{Author \MakeLowercase{\textit{et al.}}: Preparation of Papers for IEEE Control Systems Letters (August 2022)}

\title{
Balancing Morality and Economics:   Population Games with Herding and Inertia}

\author{
Raghupati Vyas,
Harsitha Devaraj
and Veeraruna Kavitha
\thanks{Department of Industrial Engineering and Operations Research, IIT Bombay, Mumbai, India. 
Emails: \texttt{raghupati.vyas@iitb.ac.in, hdevaraj@purdue.edu, vkavitha@iitb.ac.in}}
}
\maketitle
\thispagestyle{empty}
\begin{abstract}
The adoption of clean technologies (CTs) plays an important role in reducing carbon dioxide ($\coo$) emissions. We study CT adoption in a large population of consumers with heterogeneous behavioral tendencies. We model the interaction among  the agents as a multi-type mean-field game in which the agents choose between clean and polluting technology based products and may either behave as rationals (trading off price and moral incentives), herding agents (just follow the majority), or lethargic agents  exhibiting inertia toward adopting the new technologies. We characterize equilibrium CT adoption levels using the  recently introduced notion of $\balpha$-Rational Nash Equilibrium ($\balpha$-RNE) and its multi-type extension. We then identify a stable subset using the limits of a stochastic turn-by-turn behavioral dynamics. Our results highlight the role of population composition in determining CT adoption. In particular, widespread adoption requires either a sufficiently small price disadvantage for CTs or the presence of a sufficiently large herding population that can be influenced through social awareness programs. Surprisingly, we could prove that environmental damages  do not provide sufficient incentives to increase CT adoption.
\end{abstract}
\begin{IEEEkeywords}
Population games, clean technology adoption, herding behavior, 
environmental economics.
\end{IEEEkeywords}

\vspace{-1mm}
\section{Introduction}
\label{sec:introduction}
The reduction of carbon dioxide ($\coo$) emissions while sustaining economic development is a major challenge faced by modern societies. Clean technology-based solutions (briefly referred to as CTs), such as electric vehicles,   renewable energy systems have been developed to reduce these emissions and mitigate environmental damage. However, despite their long-term environmental benefits, the adoption of CTs remains limited because they often involve higher upfront costs, creating a price disadvantage compared to other alternatives.

A large body of work studies how punitive policy interventions influence the transition to CTs. 
For instance,  in \cite{acemoglu2012environment,aghion2016carbon,nordhaus2008question,newell2010induced}, the authors study how carbon taxes based policies direct firms  to switch to  CTs,  using price-based  market equilibria (at which  supply equals demand).
In \cite{carmona2022meanfield}, the authors  analyze  electricity production firms using a mean-field game framework,  where the pollution is again regulated via carbon tax imposed by the regulator.
More recently, there has been a growing interest to  study coupled dynamics, using   ordinary differential equations,  that capture evolution of environmental  changes as well as   transitions towards CT adoption among large population of firms, using mean-field    models (e.g., \cite{carmona2022meanfield})  or evolutionary replicator dynamics  (e.g., \cite{tilman2020environmental,weitz2016oscillating}). 
In all these  works, the focus is  on compelling  large rational production-units, that produce significant amounts of pollution,  to utilize CTs, via punitive taxes.

In contrast, our work shifts the focus to a much larger population but of non-atomic consumers that  voluntarily decide between CTs  (like electric vehicles, solar panels, etc)  and conventional and  lower-cost alternatives. Here, the regulator or social planner   cannot impose taxes but can instead  create moral incentives  through aggressive or continual social awareness programs.   
The per-person pollution created by such a population can be much smaller (compared to firms), however the overall effect could be significant because such a population is often  much larger in size.  Furthermore,  this huge population of consumers could exhibit a variety  of behavioral patterns while making their choices, unlike the firms, which typically resort to  more rational decisions owing to the fact that the stakes are much higher (rational firms are considered in \cite{acemoglu2012environment,aghion2016carbon,nordhaus2008question,newell2010induced,carmona2022meanfield}, while myopic rationals are considered in \cite{tilman2020environmental,weitz2016oscillating}).

We consider a large population of agents (consumers) deciding  between  CTs and unclean alternatives and exhibiting  three behavioral types:    rational agents, herding agents who just follow the majority, and lethargic agents who exhibit inertia towards  newer products,   with respective proportions given by  $\balpha = (\alpha_R,\alpha_H,\alpha_L)$. The utility function governing  rational choices depends upon the economic costs, moral incentives, and the environmental damages. 
We study the resulting interactions using a recently proposed   notion of $\balpha$- Rational Nash Equilibrium ($\balpha$-RNE) and its multi-type extension (see~\cite{agarwal2024balancing,vyas2026games,vyas2026multitype}). To identify the subset of equilibria  that are likely to emerge in practice -- we complement the static analysis with a dynamic perspective by  considering the limits of  stochastic turn-by-turn behavioral dynamics as in \cite{agarwal2025two,vyas2026multitype}. 

We have several interesting theoretical observations: 
\begin{itemize}
    \item when the lethargic agents are not too high in the population, one can achieve widespread CT adoption, even with a big price disadvantage --- but this is possible only if the herding crowd constitutes a sufficiently large fraction; 
    \item the moral incentives can be used to     compel  the entire rational crowd towards  CTs, if the proportion of the opposing (herding or lethargic) population  is not too high;  
\item however, with a large proportion exhibiting inertia, moral pressure on rational agents can also break,  leading to  zero adoption of CTs; 

\item surprisingly, 
the inclusion of a negative cost term, proportional to the environmental damage, did not alter the set of stable outcomes that form the potential limits of the stochastic dynamic process of decision adjustments;

\item there is no change in stable outcomes, even after considering rational agents that are extremely sensitive to environmental damage. 
\end{itemize} 
We begin by analyzing a game involving rational and herding agents in Section \ref{sec_bal_moral}, where the set of stable equilibria is characterized.  The extended games including the lethargic agents and then including the environmental damages are respectively considered in sections \ref{sec_extreme_moral_types} and \ref{sec_Game_influenced}. 





\hide{
We focus on stochastic behavioral dynamics in large populations, where agents update their actions through turn-by-turn interactions rather than continuous-time replicator dynamics. We assume a separation of time scales, so that environmental variables such as atmospheric $\coo$ evolve much more slowly than individual decisions and can be treated as quasi-static. This allows us to characterize the long-run outcomes (attractors, i.e., stable equilibria) arising from stochastic interactions among rational and behavioral agents.

\hide{have been used to study interlaced environmental and technological transitions in large population of production units.  For example, \cite{carmona2022meanfield} analyze carbon emission regulation and technology choices in electricity production using a mean-field framework, where a large population of rational producers responds optimally to a carbon tax imposed by a regulator, and equilibrium is characterized through mean-field equilibrium conditions.

Another line of work studies the joint evolution of strategies and environmental states through coupled dynamical systems. For example, \cite{tilman2020environmental,weitz2016oscillating} develop eco-evolutionary game frameworks in which strategy frequencies evolve via replicator dynamics coupled with environmental feedback, with environment-dependent payoffs leading to dynamical behaviors such as cycles.

where a large population of rational producers responds optimally to a carbon tax imposed by a regulator, and equilibrium is characterized through mean-field equilibrium conditions.

In these works, the focus is  on compelling the  large production-units, that produce significant amounts of pollution,  to utilize CTs.

A large body of work studies how policy 
interventions influence the transition from polluting to CTs.  For instance, \cite{acemoglu2012environment,aghion2016carbon,nordhaus2008question,newell2010induced} study how policies such as carbon taxes and  incentives, 
can mitigate environmental damage and direct technological change toward CTs. These works adopt a production-side perspective, where firms take market prices as given and adjust their decisions in response to policy-induced incentives.
 Firms interact with each other indirectly through prices and 
equilibrium determined by market-clearing conditions, under which prices adjust so that total supply equals total demand.

More recently, mean-field models have been used to study environmental and technological transitions in large populations.  For example, \cite{carmona2022meanfield} analyze carbon emission regulation and technology choices in electricity production using a mean-field framework, where a large population of rational producers responds optimally to a carbon tax imposed by a regulator, and equilibrium is characterized through mean-field equilibrium conditions.
}

\hide{
{\color{red} We mainly consider consumers who can't be taxed (like producers) when they dont shift to CT-based products. Only morality based incentives can be created to help the change among them towards CT.  We also consider the outcomes (or attractors) of stochastic dynamic decisions made by a variety of agent to study this process. }}

In this paper, we consider a large population of agents (consumers) who need to decide whether to adopt CT or an unclean alternative. The population consists of three behavioral types, namely  rational agents, herding agents who just follow the majority, and lethargic agents who exhibit inertia toward adopting new technologies, with respective fractions $\alpha_R, \alpha_H$ and $\alpha_L$ in the population, say $\balpha = (\alpha_R,\alpha_H,\alpha_L)$. The utility of rationals depends upon economic costs, moral incentives arising from social pressure, and environmental damage caused by atmospheric $\coo$ concentration. 

We model the resulting interactions using a mean-field framework and characterize the equilibrium adoption levels using the notion of $\balpha$- Rational Nash Equilibrium ($\balpha$-RNE) and its multi-type extension (see~\cite{agarwal2024balancing,vyas2026games,vyas2026multitype}). To identify which equilibria  are likely to emerge in practice -- we complement this static analysis with a dynamic perspective by adopting stochastic turn-by-turn behavioral dynamics as in \cite{agarwal2025two,vyas2026multitype}. 

We show that, when the lethargic agents are not too high in the population, one can achieve widespread CT adoption, even with a big price disadvantage,  if the herding crowd constitutes a sufficient fraction and can be influenced through effective awareness campaigns.
Morality incentives can be used to effectively   compel  the rational
crowd towards  CT, 
even if  the rest reject, when the latter proportion is not too high.  However, with a large proportion exhibiting inertia, moral pressure on rational agents can also break,  leading to  zero CT adoption.}

\hide{
The main contributions of this paper are threefold. First, we characterize the set of $\alpha$-RNEs in a technology adoption game with moral incentives and herding behavior and analyze their dependence on prices and population composition. Second, by integrating a dynamical model of atmospheric $\coo$, we study the impact of long-run environmental feedback on strategic behavior. Third, we extend the analysis to populations with heterogeneous sensitivities to environmental damages and show that heterogeneity alone is insufficient to shift equilibrium outcomes. Together, these results emphasize the central role of action-dependent incentives, such as pricing policies and moral rewards, in 
promoting the large-scale adoption of clean technologies.}

\section{A game: Balancing Morality and Economics}\label{sec_bal_moral}
We consider a large population, in which  $\alpha_R$ fraction consists of rational players, while the remaining $\alpha_H$ fraction exhibit herding behavior, i.e., they choose the action adopted by the majority, as characterized in \cite{agarwal2024balancing,vyas2026multitype,agarwal2025two,vyas2026games}; let $\balpha = (\alpha_R, \alpha_H)$; in later sections, we also consider agents that exhibit inertia towards new technology.
We consider a  game  among such a population, where the individuals choose between the products made with  clean and unclean technologies, and where the choices are guided by environmental hazards, morality perception and rational or behavioral  considerations.

Each player has to either use clean technology  (\textit{briefly referred to as CT and indicated by  action $a=1$}), or unclean technology (action $a=2$).  When $z$   fraction of the population   adopts CT, 
an agent with herding nature  (those who just follow majority) chooses CT only when $z \ge \nicefrac{1}{2}$. While  
we model the utility perceived by a typical rational agent, that drives the decisions of the agent,  
as below:
\begin{equation}\label{eqn_utility_players1}
    u(a, z) = 
    \begin{cases} 
        -P_c + (1 - z) z \, \m, & \text{if } a = 1, \\[2pt]
        -P_{uc} - (1 - z) z \, \m, & \text{if } a = 2, \mbox{ where, }
    \end{cases}
\end{equation}

$\bullet$  
  $P_c$ and $P_{uc}$ represent the prices associated with clean and unclean technologies, respectively --- the cleaner technologies  typically cost more  (e.g., electric vehicles have higher prices, living without plastic bags is highly inconvenient, etc), and these additional costs are the reason for public does not adopting them readily --- accordingly, we assume  $P_c > P_{uc}$.

$\bullet$  the term $z(1-z) \, \m$ captures the societal pressure on an individual related to morality --- the pressure is smaller either if too many are already following CT (not much pressure on the remaining few to follow) or if too few are following it (not many are following so the morality concerns are broken) --- the pressure towards morality is probably the  maximum when the society is highly divided in opinion, indicated by fraction $z$ near~$\nicefrac{1}{2}.$ 

$\bullet$ the coefficient $\m>0$ represents the trade-off between  the additional costs for CT and the morality concerns --- the higher the $\m$, the more moral the crowd is  ---   the social planner can  attempt to raise the morality coefficient in the public via meticulous advertisements or awareness-campaigns  and the result of such an effort can be captured by a bigger~$\m$.

The  strategic/behavioral  interactions that balance the morality concerns with costly, but environment friendly technologies can thus be  modeled as a mean-field game with rational utility function \eqref{eqn_utility_players1} --- observe the individual payoff depends only on the aggregate adoption level $z$;  we refer to this as \textit{Morality-guided Clean Technology adoption or  MgCT game}.

To characterize the  equilibrium behavior or outcome 
of   MgCT game \eqref{eqn_utility_players1},
in the presence of both rational and herding agents, we use the notion of $\balpha$-Rational Nash Equilibrium  ($\balpha$-RNE) recently proposed in \cite{agarwal2024balancing,vyas2026games}. For completeness, we restate the definition in our own notations and specialized to the games with two actions as below.

Towards this, we begin with some definitions. Recall  the entire herding crowd chooses the same action, that of the majority.  Thus given $z$, the fraction among the overall population that adopts CT, the fraction $y$ among the rationals that adopts CT  satisfies $z = \alpha_R y + \alpha_H \mathbf{1}_{\{z \ge \nicefrac{1}{2}\}} $. Hence, given $z$,  

\vspace{-3mm}
{\small\begin{equation}
 \hspace{-2mm}   y(z) :=  \frac{z - \alpha_H \mathbf{1}_{\{z \ge \nicefrac{1}{2}\}}}{\alpha_R}  = \frac{z}{\alpha_R} \mathbf{1}_{\{z < \nicefrac{1}{2}\}} 
    +  \frac{\alpha_R-(1 - z)}{\alpha_R}  \mathbf{1}_{\{z \geq \nicefrac{1}{2}\}},
\label{eqn_y_fun}
\end{equation}}%
represents the fraction among the rational sub-population that adopts CT. The support $\support(y)$ for  any  number   $y \in [0,1]$ (representative of a probability measure of a binary choice random variable),   denotes  the set of  actions chosen with strictly positive probability:

\vspace{-2mm}
{\small\begin{equation}\label{eqn_supp_def}
    \support(y) :=  \left \{ 
\begin{array}{lll}
       \{1,2\},  & \mbox{ if }  y \in (0,1),\\
       \{y+1\},  & \text{ else, i.e., if } y \in \{0, 1\}.
    \end{array}
    \right.
\end{equation}}
 We now reproduce
 \cite[Definition 1]{agarwal2024balancing} 
 that defines the solution, specially  for two-action games with herding population. 
%
%

\begin{definition}\label{defn_alphaRNE}
\textit{For any game with binary choices, i.e., with  ${\cal A} = \{1,2\}$, the fraction $z^*$ is called an \textbf{$\balpha$-Rational Nash Equilibrium} (\textbf{$\balpha$-RNE}), if it satisfies the following:
  \begin{eqnarray}\label{eqn_best_res_rat_play}
       \support(y(z^*)) \subseteq  {\rm Arg} \max_{i \in {\mathcal A}} u(i, z^*) \mbox{, with 
   } y(z^*)  \mbox{ as in \eqref{eqn_y_fun}}.
  \end{eqnarray}}
 \end{definition}
\medskip
Basically in our context, at any equilibrium $z^*$,  the rational players choose an action from the best response   to the aggregate CT adoption level~$z^*$, while  the herding players adopt the majority action  ---  support of the rational actions $y(z^*)$  is in \rm{Arg} $\max$ in  \eqref{eqn_best_res_rat_play}, while    the indicators in \eqref{eqn_y_fun} represent the herding choice (with tie-breaking in favor of CT action).

\vspace{-3mm}
\subsection*{Equilibria of the MgCT game}
We now determine the equilibrium  CT adoption levels  for various fractions $\alpha_R$ of the rational players, by identifying the $\balpha$-RNEs of MgCT game \eqref{eqn_utility_players1}. Towards that,   first define the following rational utility difference function using \eqref{eqn_utility_players1}:

\vspace{-3mm}
{\small \begin{eqnarray}\label{eqn_h_fun}
    h(z) :=  u(1, z) - u(2, z) = 2 (1-z) z \, \m-\Delta_P, \ 
\end{eqnarray}}
where $\Delta_P := P_c-P_{uc}$ represents the price disadvantage of choosing CT.
 The zeros of the function $h(\cdot)$  are given by

\vspace{-3mm}
{\small\begin{equation}\label{eqn_roots_h}
    \hspace{-2mm}R^-=\frac{1}{2} \left(1 - \sqrt{1 - \frac{2\Delta_P}{\m}}\right), \ 
    R^+  =\frac{1}{2} \left(1 + \sqrt{1 - \frac{2\Delta_P}{\m}} \right),
\end{equation}}%
and correspond to the interior points in $[0,1]$ at which the rational agents are indifferent between the two technologies (observe here $R^-+R^+ = 1$).  
As one may anticipate, the  roots $ R^-, R^+$  play an important role in identifying the  $\balpha$-RNEs. 

We begin with the classical case in which all the players are rational or when  $\alpha_R=1$. Using \cite[Theorem 1]{agarwal2024balancing}, 
\hide{
In this case, it is easy to verify that $z = 0$ is an $\balpha$-RNE, whereas $z = 1$ is not. To identify the remaining candidate equilibria, we characterize the zeros of the function $h(\cdot)$, which are given by

\vspace{-3mm}
{\small\begin{eqnarray}
   \hspace{-3mm} R^-\hspace{-3mm}&:=&\hspace{-3mm} \frac{1 - \sqrt{1 - 2\Delta_P/\m}}{2}, 
    R^+  := \frac{1 + \sqrt{1 - 2\Delta_P/\m}}{2}.
\end{eqnarray}}
These quantities correspond to the interior points at which rational agents are indifferent between the two technologies. }
the set of classical NEs is given by:

\vspace{-3mm}
{\small\begin{eqnarray}\label{eqn_set_rne_alpha_1}
    \N_1 = \begin{cases}
\left\{0, R^+,  R^-\right\},  & \text{ if } \Delta_P < \nicefrac{\m}{2},  \\
\{0,\nicefrac{1}{2}\}, & \text{ if } \Delta_P = \nicefrac{\m}{2}, \\
\{0\}, & \text{ if } \Delta_P > \nicefrac{\m}{2}. 
\end{cases}
\end{eqnarray}}

We next turn to the interesting case with herding   population. When the rationals constitute more than half the population  ($\alpha_R > \nicefrac{1}{2}$),   from \cite[Theorem 2]{agarwal2024balancing}, there is no change in the set of $\balpha$-RNEs,  that is  $\N_{\balpha}=\N_1$ --- basically the presence of smaller fraction of herding players does not alter the equilibrium set. 

However, with larger herding crowd (when $\alpha_R \leq \nicefrac{1}{2}$)  \textit{the set of $\balpha$-RNEs depends on the relative values of $\alpha_R$ and $\Delta_P$, and is characterized as follows} (again using \cite[Theorem 2]{agarwal2024balancing}):

\vspace{-3mm}
{\small\begin{eqnarray}\label{eqn_alpha_RNE}
\N_{\balpha} =
\begin{cases}
\left\{0,\alpha_H \right\}, 
& \hspace{-1mm}\text{if }  \Delta_P < \nicefrac{\m}{2}, \ \alpha_R < R^- , \\
\left\{0,R^-,R^+,\alpha_R \right\}, \hspace{-1mm}
&\hspace{-1mm} \text{if } \Delta_P < \nicefrac{\m}{2}, \ R^- \leq \alpha_R < \nicefrac{1}{2}, \\
\left\{ 0,\alpha_H\right\}, 
& \hspace{-1mm}\text{if }  \Delta_P \geq \nicefrac{\m}{2}.
\end{cases}
\end{eqnarray}} 
Although the preceding analysis identifies the set of $\balpha$-RNEs, it does not indicate which of these equilibria are likely to emerge in practice, given that multiple equilibria coexist.  In particular, some equilibria may be unstable and therefore are unlikely to be observed under strategic and behavioral adjustments.  We  thus complement the static equilibrium analysis with a dynamic perspective, where we identify the stable equilibria   that arise after a long period of such 
adjustments. 
%

\vspace{-2mm}
\subsection*{Stable equilibria: dynamic perspective}\label{sub_sec_turn_by_turn}
In practice, CT adoption takes place gradually over time. Individuals make adoption decisions at different time points, depending on when they become aware of new technologies, when they are able to afford them, and when they need to replace existing equipment/product. Moreover, once a person adopts a CT --- for example, by purchasing an electric vehicle or installing solar panels --- the decision typically remains in place for a long time period and is costly or inconvenient to reverse. Even among the individuals that decide against CT, only a small fraction might reconsider a change in opinion.   In summary, individuals do not frequently revise their choices with alternative options in these kinds of scenarios.

These features naturally suggest that adoption evolves through a sequence of irreversible individual decisions, influenced by the prevailing behavior in the population. To capture this process and study the stability of equilibrium adoption levels, further in the presence of herding crowd,  we consider a behavioral game dynamic in which players make decisions sequentially and only once, based on the current empirical distribution. This adjustment process is referred to as turn-by-turn dynamics in \cite{agarwal2025two,vyas2026multitype}.


We now describe the dynamics formally. Towards that, let $z_k$ denote the fraction of agents who have adopted the CT after $k$ updates. At each step $k+1$, a randomly selected agent observes the current CT adoption level $z_k$ and chooses an action $a_{k+1}$ according to its behavioral type (rational or herding).

If the agent is  rational\footnote{
The agents that consider only  current levels for decision-making, without any importance to future, are referred to as myopic rational   in   \cite{sandholm2010population,vyas2026multitype,agarwal2025two}.}, happens with probability $\alpha_R$, it chooses a best response to the current CT adoption level  $z_k$,  
\begin{equation}\label{eqn_rational_dec}
\mathbb{P}(a_{k+1}=1| \text{rational})=\mathbf 1_{\{h(z_k)\geq 0\}}, \text{ ($h$ as in   \eqref{eqn_h_fun})},
\end{equation}
where  the ties are broken in favor of  CT or action $1$.
Otherwise  the agent exhibits herding behavior and adopts the currently popular action  (with tie-breaking in favor of CT),
\begin{equation}\label{eqn_herding_dec}
\mathbb{P}(a_{k+1}=1| \text{herding})=\mathbf 1_{\{z_k\ge \nicefrac{1}{2}\}}.
\end{equation}
Accordingly, the aggregate adoption level evolves as
\begin{equation}\label{eqn_iterates_z_K}
z_{k+1}=z_k+\frac{1}{k+1}\big(\mathbf 1_{\{a_{k+1}=1\}}-z_k\big).
\end{equation}
The above  is an example of the two-choice turn-by-turn dynamics  analyzed in \cite{agarwal2025two}. 
Towards deriving its asymptotic analysis,  we begin with some notations and definitions that parallel those in \cite{agarwal2025two}. Again using \eqref{eqn_h_fun}, define
\begin{eqnarray}
\label{Eqn_Mz}
\hspace{-1mm}M(z) \hspace{-3mm}&:=&\hspace{-3mm} \mathbb{E}[\mathbf 1_{\{a_{k+1}=1\}}-z_k \mid z_k=z] \\
\hspace{-1mm}\hspace{-3mm}&=&\hspace{-3mm} \alpha_R \mathbf 1_{\{h(z) \geq 0\}} + \alpha_H \mathbf 1_{ \{z \ge \nicefrac{1}{2}\}} -z, \nonumber
\end{eqnarray}
to denote the mean drift of the process at population state $z$. Now we reproduce \cite[Definition 2]{agarwal2025two} that defines `attractors'.

\begin{definition}\label{defn_filippov}
\textit{A point $z_S^*\in[0,1]$ is called an \textbf{$\balpha$-rational attractor}, if there exists $\epsilon>0$ such that
$M(z)>0$ for all $z \in (z_S^*-\epsilon,z_S^*)$ and $M(z)<0$ for all $z \in (z_S^*,z_S^*+\epsilon)$.}
\end{definition}
\medskip
Let \textit{$\NE^S$ represent the set of  the $\balpha$-rational attractors}. 
We next prove 
  the convergence of  dynamics \eqref{eqn_iterates_z_K}:   
\begin{theorem}\label{thm:conv_Af}
\textit{Consider the behavioral dynamics  \eqref{eqn_rational_dec}-\eqref{eqn_iterates_z_K}. Then i) $z_k \to \mathcal{N}_\balpha^S \ \text{as } k\to\infty,$ almost surely; and ii) $\NE^S \subseteq \NE$, where $\N_\balpha$ is the class of $\balpha$-RNEs. }
\end{theorem}
{\bf Proof:}
Firstly,   the function $h(\cdot)$  given by  \eqref{eqn_h_fun} is continuous on $[0,1]$,  has two zeros (see \eqref{eqn_roots_h}),  and satisfies  \cite[assumption (A)]{agarwal2025two}.
Thus by  \cite[Theorems~1]{agarwal2025two}   the iterates converge almost surely to the  set  of $\balpha$-rational attractors  $\NE^S$, establishing part (i). Part (ii) follows by  \cite[Theorems~3]{agarwal2025two}. \eop
\hide{
the set  $\NE^S$
 is a subset of the class ${\cal N}_\alpha$ of the $\alpha$-RNEs given in \eqref{eqn_set_rne_alpha_1}–\eqref{eqn_alpha_RNE}, see also Definition \ref{defn_alphaRNE}. \eop}
 
\medskip
Thus, with probability one, the adoption level $z_k$ converges to the set $\NE^S \subseteq \NE$   --- we hence 
 \textit{refer   $\NE^S$  as   the set of stable equilibria}. 
Further using Definition \ref{defn_filippov} and the set of  $\balpha$-RNEs $\NE$   provided in  \eqref{eqn_set_rne_alpha_1}-\eqref{eqn_alpha_RNE}, we    explicitly characterize this stable set  in Table~\ref{table_N_alpha_S_herding_only} (basically, these are the $\balpha$-RNEs    that satisfy negative-left and positive-right sign criterion  for $M(\cdot)$).

\begin{table}[h]
    \centering
\vspace{.1mm}
\begin{tabular}{|c|c|}
\hline
\centering\textbf{Regime} & \centering\textbf{Stable set of equilibria or $\balpha$-RNEs}\arraybackslash
\\
\hline 
  $\alpha_R  \ge \nicefrac{1}{2}$   &   $
 \NE^S =     
\begin{cases}
\{0, R^+\},& \text{if }   \Delta_P<\nicefrac{\m}{2},\\
\{0\}, \hspace{-1mm}& \text{if } \ \Delta_P\ge \nicefrac{\m}{2}.
\end{cases} 
$\\\hline 
 $\alpha_R  < \nicefrac{1}{2}$      & 
  $  \NE^S  =
\begin{cases}
\{0, \alpha_H\}, \hspace{-1mm}& \text{if } \Delta_P<\nicefrac{\m}{2},\  R^+ < \alpha_H ,\\
\{0, R^+,\alpha_R\}, \hspace{-1mm}& \text{if } \Delta_P<\nicefrac{\m}{2},\ \alpha_H \leq R^+,\\
\{0,  \alpha_H\}, \hspace{-1mm}& \text{if } \Delta_P\ge \nicefrac{\m}{2}.\end{cases}$
\\ \hline 
\end{tabular}
\caption{Stable equilibria with herding. \label{table_N_alpha_S_herding_only}}
\vspace{-3mm}
\end{table}

\hide{
\begin{eqnarray} \label{eqn_stable_set_alpha_geq_half}
 \NE^S  \hspace{-2mm}&=& \hspace{-2mm}   
\begin{cases}
\{0, R^+\}, \hspace{-1mm}& \text{if }   \Delta_P<\nicefrac{\m}{2},\\
\{0\}, \hspace{-1mm}& \text{if } \ \Delta_P\ge \nicefrac{\m}{2},  
\end{cases}   \\ \nonumber  && \   \ \mbox{when  } \alpha \ge   \nicefrac{1}{2}, \mbox{ else }  \\\label{eqn_stable_set_alpha_less_half}
     \NE^S  \hspace{-2mm}&=&\hspace{-2mm}
\begin{cases}
\{0, 1-\alpha\}, \hspace{-1mm}& \text{if } \Delta_P<\nicefrac{\m}{2},\ \alpha \leq R^-,\\
\{0, R^+,\alpha\}, \hspace{-1mm}& \text{if } \Delta_P<\nicefrac{\m}{2},\ R^-<\alpha,\\
\{0, 1-\alpha\}, \hspace{-1mm}& \text{if } \Delta_P\ge \nicefrac{\m}{2}.
\end{cases}  
\end{eqnarray}}
\medskip

\begin{table*}[!b]
\centering
\footnotesize
\renewcommand{\arraystretch}{1.08}
\setlength{\tabcolsep}{3pt}
\vspace{-5mm}
\begin{minipage}{0.55\textwidth}

\begin{tabular}{|c|c|}
\hline
\centering\textbf{Regime} & \centering\textbf{Stable equilibrium}\arraybackslash
\\
\hline
 $\Delta_P \geq \nicefrac{\m}{2}$
&
$
z_S^* \in
\left\{
\begin{array}{ll}
\{0,\alpha_H\}, & \text{if } \nicefrac{1}{2} < \alpha_H,\\
\dbox{\{0\}}, & \text{otherwise}.
\end{array}
\right.
$
\\
\hline

 $\Delta_P < \nicefrac{\m}{2}$, 
 $z_S^* < \nicefrac{1}{2}$
&
$
z_S^* \in
\left\{
\begin{array}{ll}
\{0, \boxed{\alpha_R}\}, & \text{if } \alpha_L+\alpha_H<R^+,\\
\dbox{\{0\}}, & \text{otherwise}.
\end{array}
\right.
$
\\
\hline
  $\Delta_P < \nicefrac{\m}{2}$,
  $z_S^* > \nicefrac{1}{2}$
&
$
z_S^* =
\left\{
\begin{array}{ll}
\alpha_H, & \text{if } R^+ < \alpha_H,\\
R^+, & \text{if } \alpha_H \leq R^+ \leq \alpha_R+\alpha_H,\\
\boxed{\alpha_R+\alpha_H}, &
\text{if } \alpha_R+\alpha_H < R^+.
\end{array}
\right.
$
\\
\hline
\end{tabular}
\caption{Stable equilibria with herding and 
lethargic agents}
\label{table_N_alpha_S_herding_Lethargic}
\vspace{-2mm}
\end{minipage}
\hspace{-4mm}
\begin{minipage}{0.42\textwidth} 
\vspace{2mm}
\begin{tabular}{|c|c|}
\hline
\centering\textbf{Regime} & \centering\textbf{Stable equilibrium}\arraybackslash
\\
\hline
$\Delta_P \geq \nicefrac{\m}{2}$ & 
$
z_S^* \in
\left\{
\begin{array}{ll}
\{0,\alpha_H\}, & \text{if } \nicefrac{1}{2} < \alpha_H,\\
\{0\}, & \text{otherwise}.
\end{array}
\right.
$\\
\hline 
 $\Delta_P<\nicefrac{\m}{2}$,  $ z_S^* <  \nicefrac{1}{2}$   &   $
 z_S^* \in      
\begin{cases}
\{0, \boxed{\alpha_R} \}, & \text{if }     \nicefrac{1}{2} \leq \alpha_H \leq R^+,
\hspace{-1mm} \\ 
\{0\},
& \text{otherwise. } \    
\end{cases} 
$\\\hline 
$\Delta_P<\nicefrac{\m}{2}$, $z_S^*  > \nicefrac{1}{2}$      & 
  $   z_S^*  =
\begin{cases}
  \alpha_H,  \hspace{-1mm}& \text{if }  \   R^+ < \alpha_H,\\
  R^+ , \hspace{-1mm}& \text{if }   \alpha_H \leq R^+  .\end{cases}$
\\ \hline 
\end{tabular}
\caption{Stable equilibria with
  herding agents}
  \label{table_N_alpha_S_herding_L}
\end{minipage}
\end{table*}

\hide{
In particular, at $\alpha=R^-$ the equilibrium $z=\alpha$ is a knife-edge (non-attracting) point and therefore does not belong to $\mathcal N_\alpha^{\mathrm{s}}$.
Then by Theorem~\ref{thm:conv_Af} and Corollary~\ref{cor:Af_morality}, the stable equilibrium adoption levels coincide with the $\alpha$-rational attractors. Moreover, by \cite[Theorem~3]{agarwal2025two}, each limit point $z_\infty\in\mathcal A_f$ satisfies Definition~\ref{defn_alphaRNE}. Consequently, the reduced stable equilibrium set $\mathcal N_\alpha^{\mathrm{s}}$ is given by
\begin{equation}
\mathcal{N}_\alpha^{\mathrm{s}}=
\begin{cases}
\{0, R^+\}, & \text{if }\alpha>\nicefrac{1}{2},\ \Delta_P<\m/2,\\[1mm]
\{0\}, & \text{if }\alpha>\nicefrac{1}{2},\ \Delta_P\ge \m/2,\\[1mm]
\{0, 1-\alpha\}, & \text{if }\alpha\le\nicefrac{1}{2},\ \Delta_P<\m/2,\ \alpha<R^-,\\[1mm]
\{0, R^+,\alpha\}, & \text{if }\alpha\le\nicefrac{1}{2},\ \Delta_P<\m/2,\ R^-<\alpha<\nicefrac{1}{2},\\[1mm]
\{0, 1-\alpha\}, & \text{if }\alpha\le\nicefrac{1}{2},\ \Delta_P\ge \m/2.
\end{cases}
\end{equation}}
\hide{
Furthermore, by \cite[Theorem~3]{agarwal2025two}, each limit point $z_\infty\in\mathcal A_f$ satisfies the Definition \ref{defn_alphaRNE} of $\alpha$-RNE. Thus the reduced set of equilibria (denoted by $\NEc^{\mathrm{s}}$) is given as below:}
\hide{

We now obtain the stable subset of \eqref{eqn_set_rne_alpha_1}-\eqref{eqn_alpha_RNE}, using the results of \cite{agarwal2025two}. Towards this,  we begin with some notations and definitions. 

Let
\begin{equation}
M(z):=\mathbb{E}[a_{k+1}-z_k\mid z_k=z]
\end{equation}
denote the mean drift of the process at state $z$.

\begin{definition}\label{defn_filippov}
A point $z_\infty\in[0,1]$ is called an \textbf{$\alpha$-rational attractor} if there exists $\epsilon>0$ such that
$\mathrm{sign}\{(z-z_\infty)M(z)\}$ is negative for all
$z\in(z_\infty-\epsilon,z_\infty+\epsilon)\setminus\{z_\infty\}$.
\end{definition}
Let $\mathcal A_f$ denote the set of all $\alpha$-rational attractors.

Observe that the utility difference function $h(\cdot)$ defined in \eqref{eqn_h_fun} is continuous on $[0,1]$ and has finitely many zeros with sign changes in every neighborhood. Hence, Assumption (A) in \cite{agarwal2025two} is satisfied. The following convergence result therefore applies.

We first introduce some notations, as in \cite{agarwal2025two}, to describe the dynamics formally. Let $R_k$ be the indicator that the $k$-th player is myopic rational, and let $F_k$ or $G_k$ denote the indicators that the $k$-th player chooses CT (action $1$) when being myopic rational or herding, respectively. The fraction of agents adopting the CT after $k$ updates equals,
\begin{eqnarray}\label{eqn_iterates_z_K}
    z_k=\frac{1}{k}\sum_{i=1}^k \left(R_i F_i+(1-R_i)G_i\right).
\end{eqnarray}
If the $(k+1)$-th player is myopic rational, then it chooses the action that maximizes its instantaneous utility, i.e.,
\begin{eqnarray}
    \mathbb P(F_{k+1}=1)\hspace{-2mm}&=&\hspace{-2mm}\mathbf 1_{\{1=a^R(z_k)\}}, \text{ where } \label{eqn_rational_dec}\\
a^R(z)\hspace{-2mm}&:=&\hspace{-2mm}\min \left\{i:\, i\in\arg\max_{a\in\mathcal A} u(a,z)\right\},
\end{eqnarray}
where  the ties are broken in favor of action $1$. If the $(k+1)$-th player exhibits herding behavior, then it adopts the currently popular action:
\begin{eqnarray}\label{eqn_herding_dec}
    \mathbb P(G_{k+1}=1\mid R_{k+1}=0)
=\mathbf 1_{\{z_k\ge \nicefrac{1}{2}\}}.
\end{eqnarray}
The stochastic process $\{z_k\}_{k\ge 0}$ describes the evolution of
aggregate adoption under sequential play.  
Also from \eqref{eqn_iterates_z_K}, the iterates for $z_k$ can be re-written as below, 
\begin{align}\label{eqn_SA}
    z_{k+1} &= z_k + \frac{1}{k+1}g(R_{k+1}, F_{k+1}, G_{k+1}, z_k), 
    \\
    &\mbox{ where } g(R, F, G, z) := RF+(1-R)G-z.
\end{align}
The above iterative dynamics therefore resemble those analyzed in \cite{agarwal2025two}. We now obtain the stable subset of \eqref{eqn_set_rne_alpha_1}-\eqref{eqn_alpha_RNE}, using the results of \cite{agarwal2025two}. Towards this,  we begin with some notations and definitions. 

\begin{definition}\label{defn_filippov}
 \textit{   A point $z_\infty \in [0,1]$ is called an \textbf{$\alpha$-rational attractor} if for some $\epsilon > 0$, $\mbox{sign}\big\{(z-z_\infty)M(z)\big\}$ is negative for all $z \in (z_\infty - \epsilon, z_\infty + \epsilon)-\{z_\infty\}$, where 
\begin{eqnarray}
M(z):=E[g(R_{k+1}, F_{k+1}, G_{k+1}, z_k)|z_k = z].
\end{eqnarray}}
\end{definition}
\medskip
Let  ${\cal A}_f$ denote the set of all possible $\alpha$-rational attractors. 

Observe that the function $h(\cdot)$ in \eqref{eqn_h_fun} is continuous on $[0,1]$ and has finitely many zeros, with a change of sign in every neighborhood of each zero. Hence \cite[assumption (A)]{agarwal2025two} is satisfied. Then the following theorem is a direct consequence of \cite[Theorem 1]{agarwal2025two}. 
\begin{theorem}\label{thm:conv_Af}
Consider the behavioral dynamics described in \eqref{eqn_iterates_z_K}-\eqref{eqn_herding_dec}. Then $z_k \to \mathcal A_f \quad \text{as } k\to\infty,$ almost surely. \eop
\end{theorem}

Now we identify the set $\mathcal A_f$  using \cite[Theorem 4]{agarwal2025two} and summarize the result in the following corollary.
\begin{corollary}\label{cor:Af_morality}
The following statements hold for the above game.
\begin{itemize}
\item[(i)] If $\alpha \geq \nicefrac{1}{2}$, then
$$
\mathcal{A}_f  =
\begin{cases}
\{0,R^+\}, &\text{if } \Delta_P< \nicefrac{\m}{2},\\
\{0\}, &\text{if } \Delta_P \geq \nicefrac{\m}{2}.
\end{cases}
$$

\item[(ii)] If $\alpha < \nicefrac{1}{2}$, then $0\in \mathcal A_f$ and
\begin{eqnarray*}
  \hspace{-2mm}  S\setminus(\alpha,1-\alpha) \hspace{-2mm}&\subseteq& \hspace{-2mm}\mathcal A_f \subseteq
\big(S\setminus(\alpha,1-\alpha)\big)\cup\{\alpha,1-\alpha\}, \text{ where }\\
 \hspace{-2mm} S \hspace{-2mm}&\subseteq& \hspace{-2mm} \begin{cases}
\{0,R^+\}, &\text{if } \Delta_P< \nicefrac{\m}{2},\\
\{0\}, &\text{if } \Delta_P \geq \nicefrac{\m}{2}.
\end{cases}
\end{eqnarray*}
Moreover,
\begin{itemize}
    \item[(a)] $\alpha\in\mathcal A_f \ \text{only if}\ \Delta_P<2\alpha(1-\alpha)\m,$
    \item[(b)] $1-\alpha\in\mathcal A_f \ \text{only if}\ \Delta_P>2\alpha(1-\alpha)\m.$ \eop
\end{itemize}
\end{itemize}
\end{corollary}
Furthermore, by \cite[Theorem~3]{agarwal2025two}, each limit point $z_\infty\in\mathcal A_f$ satisfies the Definition \ref{defn_alphaRNE} of $\alpha$-RNE. Thus the reduced set of equilibria (denoted by $\NEc^{\mathrm{s}}$) is given as below: }
\hide{
The above iterative equation resembles a stochastic approximation (SA) scheme (see, e.g., \cite{borkar2009stochastic, kushner2003stochastic}). Such iterates typically converge to the attractors of the associated differential inclusion (DI), see \cite{filippov1960differential}, given by
\begin{eqnarray}\label{eqn_DI}
    \dot{z} \in M(z):=E[g(R_{k+1}, F_{k+1}, G_{k+1}, z_k)|z_k = z].
\end{eqnarray}
Evaluating the conditional expectation gives
\begin{eqnarray}
M(z)\hspace{-3mm}&=&\hspace{-3mm} \alpha  1_{\{u(1, z) \ge u(2, z)\}} + (1-\alpha)1_{\{z \geq \nicefrac{1}{2}\}}-z, \\
\hspace{-3mm}&=&\hspace{-3mm} \alpha 1_{\{(z-R^-)(z-R^+) \leq 0\}} + (1-\alpha)1_{\{z \geq \nicefrac{1}{2}\}}-z.
\end{eqnarray}
 Finally, the following result,
adapted from \cite{agarwal2025two}, characterizes its long-run behavior of dynamics.

The above theorem shows that the long-run behavior of the stochastic process is governed by the set of $\alpha$-rational attractors of the associated DI, which is formally defined as below.
\begin{definition}\label{defn_filippov}
    A point $z_\infty \in [0,1]$ is called an \textbf{$\alpha$-rational attractor} if for some $\epsilon > 0$, $\mbox{sign}\big\{(z-z_\infty)M(z)\big\}$ is negative for all $z \in (z_\infty - \epsilon, z_\infty + \epsilon)-\{z_\infty\}$.
\end{definition}
We next characterize the set $\mathcal A_f$ of $\alpha$-rational attractors associated with the DI \eqref{eqn_DI}, using \cite[Theorem 4]{agarwal2025two} and summarize the result in the following corollary.
\begin{corollary}\label{cor:Af_morality}
The following statements hold for the above game.
\begin{itemize}
\item[(i)] If $\alpha \geq \nicefrac{1}{2}$, then
$$
\mathcal{A}_f  =
\begin{cases}
\{0,R^+\}, &\text{if } \Delta_P< \nicefrac{\m}{2},\\
\{0\}, &\text{if } \Delta_P \geq \nicefrac{\m}{2}.
\end{cases}
$$

\item[(ii)] If $\alpha < \nicefrac{1}{2}$, then $0\in \mathcal A_f$ and
\begin{eqnarray*}
  \hspace{-2mm}  S\setminus(\alpha,1-\alpha) \hspace{-2mm}&\subseteq& \hspace{-2mm}\mathcal A_f \subseteq
\big(S\setminus(\alpha,1-\alpha)\big)\cup\{\alpha,1-\alpha\}, \text{ where }\\
 \hspace{-2mm} S \hspace{-2mm}&\subseteq& \hspace{-2mm} \begin{cases}
\{0,R^+\}, &\text{if } \Delta_P< \nicefrac{\m}{2},\\
\{0\}, &\text{if } \Delta_P \geq \nicefrac{\m}{2}.
\end{cases}
\end{eqnarray*}
Moreover,
\begin{itemize}
    \item[(a)] $\alpha\in\mathcal A_f \ \text{only if}\ \Delta_P<2\alpha(1-\alpha)\m,$
    \item[(b)] $1-\alpha\in\mathcal A_f \ \text{only if}\ \Delta_P>2\alpha(1-\alpha)\m.$ \eop
\end{itemize}
\end{itemize}
\end{corollary}
Furthermore, by \cite[Theorem~3]{agarwal2025two}, each limit point $z_\infty\in\mathcal A_f$ satisfies the Definition \ref{defn_alphaRNE} of $\alpha$-RNE . Thus, every $\alpha$-rational attractor corresponds to an $\alpha$-RNE.

Finally, we relate the attractor characterization to the equilibrium sets in
\eqref{eqn_set_rne_alpha_1}-\eqref{eqn_alpha_RNE}. Since the equilibria $0$ and $R^-$ are unstable, we define the reduced set of stable equilibria
by removing these unstable equilibria from $N_\alpha$ and denote it by
$N_\alpha^{\mathrm{s}}$.}
\hide{
From \eqref{eqn_set_rne_alpha_1}, the reduced stable set of classical equilibria is
$$
\NEc^{\mathrm{s}}=
\begin{cases}
\{0, R^+\}, & \text{if }\Delta_P<\m/2,\\
\{0\}, & \text{if }\Delta_P \geq \m/2.
\end{cases}
$$
We know that $\mathcal{N}^{\mathrm{s}}_\alpha = \NEc^{\mathrm{s}}$ for all $\alpha > \nicefrac{1}{2}$. From \eqref{eqn_alpha_RNE}, for $\alpha\le \nicefrac{1}{2}$ the reduced stable equilibrium set is
\begin{eqnarray}
   \mathcal{N}_\alpha^{\mathrm{s}}=
\begin{cases}
\{0, 1-\alpha\}, & \text{if }\Delta_P<\m/2,\ \alpha<R^-,\\
\{0, R^+,\alpha\}, & \text{if }\Delta_P<\m/2,\ R^- < \alpha< \nicefrac{1}{2},\\
\{0, 1-\alpha\}, & \text{if }\Delta_P\ge \m/2.
\end{cases} 
\end{eqnarray}
In particular, at $\alpha=R^-$ the equilibrium $z=\alpha$ is a knife-edge (non-attracting) point and therefore does not belong to $\mathcal N_\alpha^{\mathrm{s}}$.}
\hide{
{\color{blue} 
\begin{prop}\label{prop:eq_three_types_cases} The stable equilibrium adoption set $\mathcal N^S_{\alpha_r}$ is given as:
\noindent\textbf{(i) If $\Delta_P>\nicefrac{\m}{2}$.}
$$
\mathcal N^S_{\alpha_r}=
\left\{
\begin{array}{ll}
\alpha_m, & \text{if } \alpha_m<\nicefrac{1}{2},\\
\alpha_m+\alpha_h, & \text{if } \alpha_r < \nicefrac{1}{2}.
\end{array}
\right.
$$
\noindent\textbf{(ii) If $\Delta_P=\nicefrac{\m}{2}$.}
$$
\mathcal N^S_{\alpha_r}=
\left\{
\begin{array}{ll}
\alpha_m, & \text{if } \alpha_m<\nicefrac{1}{2},\\
\alpha_m+\alpha_h, & \text{if } \alpha_r<\nicefrac{1}{2},.
\end{array}
\right.
$$
\noindent\textbf{(iii) If $\Delta_P<\nicefrac{\m}{2}$.}
$$
\mathcal N^S_{\alpha_r}=
\left\{
\begin{array}{ll}
\alpha_m, & \text{if }  \alpha_m < R^-,\\
\alpha_m+\alpha_r, & \text{if }\ R^-<\alpha_m+\alpha_r < \nicefrac{1}{2},\\
\alpha_m+\alpha_h, & \text{if } \alpha_r < R^-,\\
R^+, & \text{if } \alpha_r > R^-.
\end{array}
\right.
$$
Moreover, there is no equilibrium with 
$z\in\big[\nicefrac{1}{2},R^+\big).$
\end{prop}

\begin{prop}\label{prop:eq_three_typ_dem}
Assume $\alpha_r+\alpha_h+\alpha_{um}=1$.
Then the stable equilibrium adoption set $\mathcal N_{\alpha_r}^S$ is given as:

\noindent\textbf{(i) If $\Delta_P>\nicefrac{\m}{2}$.}
$$
\mathcal N_{\alpha_r}^S =
\left\{
\begin{array}{ll}
0, & \text{always},\\
\alpha_h, & \text{if } \alpha_r+\alpha_{um} < \nicefrac{1}{2}.
\end{array}
\right.
$$
\noindent\textbf{(ii) If $\Delta_P=\nicefrac{\m}{2}$.}
$$
\mathcal N_{\alpha_r}^S=
\left\{
\begin{array}{ll}
0, & \text{always},\\
\alpha_h, & \text{if } \alpha_r+\alpha_{um} < \nicefrac{1}{2}.
\end{array}
\right.
$$
\noindent\textbf{(iii) If $\Delta_P<\nicefrac{\m}{2}$.}
$$
\mathcal N_{\alpha_r}^S=
\left\{
\begin{array}{ll}
0, & \text{always},\\
\alpha_r, & \text{if } R^-<\alpha_r<\nicefrac{1}{2},\\
\alpha_r+\alpha_h, & \text{if } R^- < \alpha_{um} < \nicefrac{1}{2},\\
\alpha_h, & \text{if } \alpha_{um}+\alpha_r < R^-,\\
R^+, &\text{if } \alpha_{um}\le R^-<\alpha_{um}+\alpha_r.
\end{array}
\right.
$$
\end{prop}}}

{\bf Remarks:} 
When the price disadvantage of CT is sufficiently high (with $\Delta_P \ge \nicefrac{\m}{2}$), and the population contains a large fraction of rationals (with $\alpha_R \ge \nicefrac{1}{2}$), then nobody adopts CT --- observe in the first row of Table~\ref{table_N_alpha_S_herding_only}, the only stable $\balpha$-RNE is $0$. 
However if the population is composed of a larger fraction of herding crowd (with  $\alpha_R < \nicefrac{1}{2}$), there is a chance to successfully promote CT --- in the second row, we have    $\N_\alpha^S = \{0,\alpha_H\}$  when $\Delta_P \ge \nicefrac{\m}{2}$. 

Nonetheless $0$ is still a stable equilibrium, and   the social planner should work towards emergence  of (or convergence to)   desirable  $\balpha$-RNE, that of $\alpha_H$ --- this may be possible by aggressive advertisements, rapid awareness programs, etc., which can propel the herding crowd towards more desirable CT  adoption (see \cite{vyas2026games} for similar design details  with herding crowd). Thus \textit{if herding behavior is predominant in the population, there is a possibility to  successfully make a substantial population embrace  CT, in spite of huge price disadvantage}. 

When \textit{the price disadvantage of CT is not too large} (with $\Delta_P < \nicefrac{\m}{2}$), the moral incentive becomes strong enough to potentially offset the extra cost of CT. In this regime, even with all rational players  ($\alpha_R=1$), a non-zero value  $R^+ $ is one of the stable equilibria  --- the stable set in the first row of Table~\ref{table_N_alpha_S_herding_only} is   $\NE^S=\{0,R^+\}$. However,  from \eqref{eqn_roots_h}, the magnitude of $R^+$ is  inversely proportional to  the ratio, $\nicefrac{\Delta_P}{\m}$. Thus, when the additional cost of CT remains significant, the achievable CT adoption level with purely rational  population is limited -- in fact this is true  even when the population is dominated by rationals (more than $50\%$).  

Once again, the presence of a substantial herding crowd ($\alpha_R \leq \nicefrac{1}{2}$) can significantly alter this outcome -- the equilibrium $R^+$ may be replaced by $\alpha_H$, enabling the possibility of  a much higher level of CT adoption.

Thus in all,  to mobilize a large   fraction   towards  CT:
\textit{\begin{itemize}
    \item either the price disadvantage is sufficiently small,
    \item or the population     contains a substantial fraction of  herding crowd and an effective awareness campaign and/or aggressive advertisements can be conducted. 
\end{itemize}
}

In this section, we analyzed a scenario with only herding and rational agents. But one may notice many other behavioral patterns. Another relevant predominant behavior is  inertia or adherence to old techniques. We next consider the same.


\section{Agents with inertia}\label{sec_extreme_moral_types}
In transition environments like the one studied here, inertia often manifests as resistance to adopting CT, which are typically more expensive than the incumbent alternatives. As a result, some individuals continue using the unclean technology despite the presence of moral or social incentives. Motivated by this, 
we   consider a more detailed study by including   agents with such inertia\footnote{One can consider a study with  strongly moral agents that  would readily accept  CT. The results were not much different from those in previous section and due to lack of space, we could not include this study. } and  refer to them as \textit{`lethargic or $L$ agents'}.

The population now consists of three types of agents: rational, herding, and lethargic, let  $\Theta = \{R,H,L\}$ represent these types. Let the corresponding proportions be $\alpha_R$, $\alpha_H$, and $\alpha_{L}$, respectively,  and let $\balp := (\alpha_R, \alpha_H, \alpha_L)$.
We  now adopt the multi-type mean field game model developed in \cite{vyas2026multitype} to analyze dynamics as in \eqref{eqn_iterates_z_K}, but now including $L$ agents and towards that one needs to capture (if possible) the choices of all the types through maximizing some type-wise utility functions.  It is easy to define such utility functions:  
the utility of type $R$ agents is already given by \eqref{eqn_utility_players1}, while that of  types $H$ and $L$ can be captured by

\vspace{-3mm}
{\small\begin{equation}\label{eqn_util_herd_leth}
\hspace{-2mm}u_H(a,z) =
\begin{cases}
z, & \hspace{-2.5mm}\text{if }a=1,\\
1-z, &\hspace{-2.5mm}\text{if } a=2,
\end{cases} 
\ \ \ u_{L}(a,z) =
\begin{cases}
0, &\hspace{-2.5mm}\text{if } a=1,\\
1, &\hspace{-2.5mm}\text{if } a=2.
\end{cases}
\end{equation}}%
(observe for example  that the  herding crowd chooses the action  of the majority and hence chooses CT if and only if $z  > (1-z)$, which is precisely captured by $u_H$ function).

We  utilize the equilibrium notion   of \cite{vyas2026multitype} to study this game. To this end, we reproduce \cite[Definition~1]{vyas2026multitype} in our notations. 
 \begin{definition}[MT-AMFE]\label{defn_MT-MFE_binary} 
 We call $z^*$ a multi-type aggregate mean-field equilibrium (MT-AMFE) if it satisfies: \begin{itemize}
     \item $z^* := \sum_{\theta\in\Theta}\alpha_\theta \mu_\theta^*$, where $\mu_\theta^*$, for each $\theta \in \Theta$,  is the fraction of type $\theta$ agents adopting CT;
     \item the choices at equilibrium are    type-wise optimal, i.e.,
     \begin{equation}\label{eqn_MT_MFE_BR}
     \hspace{-1mm}
\support(\mu_\t^*) \subseteq {\rm Arg} \max_{a\in\A_\t} u_\t(a,z^*),  \mbox{ for each } \theta.
\end{equation}
 \end{itemize}
\end{definition}

\hide{
\subsection*{Multi-type aggregate mean-field equilibrium (MT-AMFE)}
Let $\Theta$ be a finite set of player types, and let $\alpha_\theta$ be the fraction of players of type $\theta$ in the population, with $\sum_{\t\in\Theta}\alpha_\t=1$. Each type $\t$ chooses
an action from $\A :=\{1,2\}$. Let $\bmu_\t = (\mu_\theta,1-\mu_\theta)$ denote the 
empirical distribution over actions of type-$\t$ agents, where $\mu_\theta$ represents the fraction of type-$\theta$ players adopting CT.

Define the aggregate population measure $\bm z = (z,1-z)$ by
\begin{equation}\label{eqn_agg_measure_binary}
\bm z := \sum_{\t\in\Theta}\alpha_\t \bmu_\t.
\end{equation}
where $z$ represents the aggregate CT adoption level.
The utility of a type-$\t$ player choosing action $a\in\A$ depends on the aggregate state only through $z$, and is denoted by $u_\t(a,z)$.

\begin{definition}[MT-AMFE]\label{defn_MT-MFE_binary} 
Consider a  game  with  ${\cal A} = \{1,2\}$.
A tuple $(\mu_\t^*)_{\t\in\Theta}$ is called a multi-type mean-field Nash equilibrium (MT-MFE) if, for every $\t\in\Theta$, (see \eqref{eqn_supp_def})
\begin{equation}\label{eqn_MT_MFE_BR}
\support(\mu_\t^*) \subseteq \arg\max_{a\in\A_\t} u_\t(a,z^*), \text{ with } z^* := \sum_{\theta\in\Theta}\alpha_\theta \mu_\theta^*.
\end{equation}
We call $z^*$ a multi-type aggregate mean-field equilibrium (MT-AMFE).
\end{definition}}
\medskip
\noindent
Observe that when $\Theta=\{R,H\}$, the MT-AMFE defined above coincides with the $\alpha$-RNE introduced in Definition~\ref{defn_alphaRNE}.
\hide{
Let the population be composed of rational, herding, and lethargic agents with proportions $\alpha_R$, $\alpha_H$, and $\alpha_{L}$, respectively, satisfying $\alpha_R + \alpha_H + \alpha_{L} = 1$ (let $\balp = (\alpha_R, \alpha_H, \alpha_L)$). Accordingly, $\Theta = \{R,H,L\}$. The utility of type $R$ agents is given by \eqref{eqn_utility_players1}, while the utilities of types $H$ and $L$ are specified by
\begin{equation*}
\begin{aligned}
u_H(a,z) &=
\begin{cases}
z, & \text{if }a=1,\\
1-z, &\text{if } a=2,
\end{cases}
u_{L}(a,z) =
\begin{cases}
0, &\text{if } a=1,\\
1, &\text{if } a=2.
\end{cases}
\end{aligned}
\end{equation*}}
One can characterize the equilibria (i.e., MT-AMFEs) for the game with lethargic agents using Definition~\ref{defn_MT-MFE_binary}. For brevity, we do not list all MT-AMFEs explicitly here; instead, we directly present the subset of stable equilibria.

To once again study the stability of MT-AMFEs from a dynamic perspective, we consider the turn-by-turn dynamics (as discussed in Subsection~\ref{sub_sec_turn_by_turn}), extended to the current case. As shown in \cite[Theorems~2 and~5, Definition~4]{vyas2026multitype}, this process converges almost surely to a singleton internally chain transitive (ICT) set, which is  an MT-AMFE.

Since the construction of the turn-by-turn dynamics in the multi-type setting follows the same principles as in the two-type case (see \eqref{eqn_iterates_z_K}), we do not repeat it here. Furthermore, to establish the stability of an MT-AMFE, we use Definition~\ref{defn_filippov}, as in the two-type framework. In particular, for the MgCT game with lethargic agents, define
\begin{equation}\label{eqn_M_I_fun}
    M^I (z) :=  \alpha_R \mathbf 1_{\{h(z) \geq 0\}} + \alpha_H \mathbf 1_{ \{z \ge \nicefrac{1}{2}\}} -z.
\end{equation}
We say \textit{$z^{*}_S$  is a stable  equilibrium adoption level if it is an MT-AMFE (satisfies \eqref{eqn_MT_MFE_BR}) and  is an attractor  as in Definition~\ref{defn_filippov}, with $M(\cdot)=M^I(\cdot)$.} The stable equilibria are summarized  in Table \ref{table_N_alpha_S_herding_Lethargic}; in particular,   $z_S^*$ is a stable equilibrium iff it satisfies  the
conditions listed there (the computations are in the Appendix). Table \ref{table_N_alpha_S_herding_only} is rewritten as   Table \ref{table_N_alpha_S_herding_L}, for ease of comparing the scenarios with and without lethargic $L$ agents.




\hide{
\begin{table}[h]
    \centering
\vspace{3mm}
\begin{tabular}{|c|c|}
\hline 
  $\alpha_R  \ge 1/2$   &   $
 \NE^S =     
\begin{cases}
\{0, R^+\},& \text{if }   \Delta_P<\nicefrac{\m}{2},\\
\{0\}, \hspace{-1mm}& \text{if } \ \Delta_P\ge \nicefrac{\m}{2},  
\end{cases} 
$\\\hline 
 $\alpha_R  < 1/2$      & 
  $  \NE^S  =
\begin{cases}
\{0, \alpha_H\}, \hspace{-1mm}& \text{if } \Delta_P<\nicefrac{\m}{2},\  R^+ < \alpha_H ,\\
\{0, R^+,\alpha_R\}, \hspace{-1mm}& \text{if } \Delta_P<\nicefrac{\m}{2},\ \alpha_H \leq R^+,\\
\{0,  \alpha_H\}, \hspace{-1mm}& \text{if } \Delta_P\ge \nicefrac{\m}{2}.\end{cases}$
\\ \hline 
\end{tabular}
\caption{Stable equilibria with herding, $\alpha_R  = \alpha$, $\alpha_H=1-\alpha$ \label{table_N_alpha_S_herding_only}}
\vspace{-15mm}

\begin{tabular}{|c|c|}
\hline 
  $ z_S^* <  \nicefrac{1}{2}$   &   $
 z_S^* \in      
\begin{cases}
\{0, \alpha_R \}, & \text{if }     \nicefrac{1}{2} \leq \alpha_H \leq R^+,
\hspace{-1mm} \\ 
\{0\}
& \text{else } \    
\end{cases} 
$\\\hline 
 $z_S^*  > 1/2$      & 
  $   z_S^*  =
\begin{cases}
  \alpha_H  \hspace{-1mm}& \text{if }  \   R^+ < \alpha_H,\\
  R^+ , \hspace{-1mm}& \text{if }   \alpha_H \leq R^+  .\end{cases}$
\\ \hline 
\end{tabular}
\caption{With $\Delta_P<\nicefrac{\m}{2} $}
\vspace{-15mm}
\end{table}}
Most of the implications are as in the previous case without $L$ agents. However, there is one interesting distinction. From the last row of Table \ref{table_N_alpha_S_herding_Lethargic}, one can notice that all the rational agents choose CT at equilibrium when $\alpha_R +\alpha_H < R^+$; in fact even in the second row of both the tables \ref{table_N_alpha_S_herding_Lethargic} and \ref{table_N_alpha_S_herding_L},  all the rational agents   choose CT adoption, when the $H$ or $H+L$ agents choose the other alternative (see the boxes). Hence, in many scenarios, driven by morality incentives, the rational agents are compelled to choose CT  when a considerable fraction of  others choose the other-way.

However, such a compulsion breaks when there are too many $L$-agents and few rationals. 
When $\alpha_L > \nicefrac{1}{2}$ and $R^+ \leq \alpha_L+\alpha_H$, we have  $\N^S = \{0\}$  (see dashed boxes in Table \ref{table_N_alpha_S_herding_Lethargic}, also observe none of the remaining cases including the third row are possible with $\alpha_R+\alpha_H < \nicefrac{1}{2}$) --- this is probably because, when $L$ agents are in majority,  the  herding crowd follows them and  the moral pressure on  rational agents reduces significantly  (term $(1-z)z \m$ in \eqref{eqn_utility_players1} is small when $z$ is near~0).    
\hide{
\sout{Furthermore, lethargic agents act as a structural barrier that limits the achievable level of CT adoption. In particular, when lethargic agents dominate the population, i.e., $\alpha_L > \nicefrac{1}{2}$ and $R^+ \leq \alpha_L+\alpha_H$, then no one adopts CT (see Table~\ref{table_N_alpha_S_herding_Lethargic}) --- when lethargic agents dominate, herding players follow them, and since the CT adoption level $z$ remains too small for the  social (moral) pressure $\m z (1-z)$ in \eqref{eqn_utility_players1} to offset the price disadvantage of CT, rational agents prefer the unclean technology. Hence the entire population ends up using the unclean technology.}}

So far, we have studied how varieties of population respond to the availability of CTs when decisions are driven by prices, moral incentives, and behavioral tendencies. However, the primary motivation for adopting CTs is to limit the environmental hazards caused by pollution, particularly the accumulation of atmospheric $\coo$. We now investigate \textit{how the outcome of the game changes when (rational) agents     predict the impact of their collective choices on atmospheric $\coo$ concentration and consider the same while making decisions}. 

Towards that, we next extend the population game  by incorporating the effects of $\coo$ concentration into the utility function \eqref{eqn_utility_players1}. The goal here is to analyze if such a consideration  can create sufficient  incentives to improve  CT adoption. 

\section{Game influenced by Environmental effects }\label{sec_Game_influenced}
We again consider three types of agents, the choices of $H$ and $L$   agents remain as before.
Towards altering the rational utility 
 function \eqref{eqn_utility_players1}, 
we next  discuss  the evolution of atmospheric
$\coo$ concentration, 
 represented by $c(t)$ at time~$t$. 
This evolution  is modeled using an ordinary differential equation (ODE) driven by a function $f(c;z)$  that also depends upon $z$, the   CT adoption level of the population (\cite{caetano2008optimal, misra2013mathematical,joos2013carbon}):

\vspace{-2mm}
{\small\begin{equation}\label{eqn_general_ODE}
    \frac{dc}{dt} = f(c;z), \text{ with } c(0) > 0.
\end{equation}}
Our results of this section are valid under minimal assumptions:  (i)  a Lipschitz function $f$
 that ensures the existence of ODE solution on the interval of interest (say   $T$);  and  ii) $f$ is non-increasing in $z$ to rightfully represent the impact of population choices (basically  higher
adoption of CT reduces emission).  There is a vast literature that studies the evolution of $\coo$   (see e.g., \cite{caetano2008optimal, misra2013mathematical,joos2013carbon}), any of these  models can be used to define $f $ in \eqref{eqn_general_ODE} after appropriately introducing the influence of $z$ --- for example, in the $\coo$ evolution model  \cite[equation (1)]{misra2013mathematical}, the term  $N$  representing the human population size  can be modified as $N(1-z)$ in $\nicefrac{dX}{dt}$ (the derivative of $\coo$ evolution) to indicate the effective population size  that adversarially influences the $  \coo$ evolution.

\hide{ 
\subsection{Atmospheric $\coo$ Dynamics}
Let $C_c(t)$ denote the atmospheric $\coo$ concentration at time $t$. The evolution of $C_c$ depends on the aggregate adoption level $z$ of CTs in the population. Hence the dynamics of $C_c$ can be described by
\begin{equation}\label{eqn_general_ODE}
    \frac{dC_c}{dt} = f(C_c;z), \text{ with } C_c(0) > 0.
\end{equation}
We assume that $f$ is non-increasing in $z$, reflecting that higher adoption of CT reduces emissions. Many climate models used in the literature \cite{} fall within this formulation. 
For instance, in the $\coo$ evolution model of \cite{}, the human population size $H_n$ can be modified as $H_n(1-z)$, leading to
$$
f(C_c;z,q,\gamma,H_n,\gamma_0) = q + \gamma H_n(1-z) - \gamma_0 C_c.
$$

We require the existence of a solution for \eqref{eqn_general_ODE} for each $z$ on the required horizon $T$ (can also be infinite). 
}

We assume the rationals perceive a negative environmental cost $e(z)  $ which  is some  function   of $z$-influenced  $\coo$ trajectory  $\{c(t; z) \}_{t\in T}$  --- for example, it could be due to average predicted discomfort   endured during the  period $T$  captured   by $ e (z) = \nicefrac{1}{T}\int_T \phi(t,c(t;z)) dt$ for some function $\phi$, or it could be due to the influence of the long run concentration captured by $ e (z) = \lim_{t\to \infty} \phi(t, c(t;z))$ (when $T = [0, \infty))$, etc.


Under the monotonicity assumption on~$f$, using standard ODE results, \textit{one can assume $e(z)$ is non-increasing in $z$}. 




Now consider the modified population game where only rational utility \eqref{eqn_utility_players1} changes as below (for $\rho >0$):
%

\vspace{-3mm}
{\small\begin{equation}\label{utility_players}
    u^{E}_R(a, z) = \begin{cases}
    -P_c + (1 - z) z \, \m - \rho e (z),& \hspace{-1mm}\text{if } a = 1,\\[2pt]
    -P_{uc} - (1 - z) z \, \m - \rho e (z),& \hspace{-1mm}\text{if } a = 2.
\end{cases}
\end{equation}}%
In the above, the term $\rho e (z)$ is incorporated symmetrically 
because $\coo$ affects all individuals regardless of the technology they adopt. The utilities  of $H$ and $L$  agents remain the same as in \eqref{eqn_util_herd_leth}. 
We immediately make a striking observation ---  the utility difference function, obtained now using \eqref{utility_players}, is:
\begin{equation}\label{eqn_h_c_fun}
    h^E(z) :=  u^E_R(1, z) - u^E_R(2, z) = 2 (1-z) z \, \m-\Delta_P,
\end{equation}
which exactly coincides with \eqref{eqn_h_fun}. Hence, there is absolutely no change in the set of stable MT-AMFEs, see \eqref{Eqn_Mz}, \eqref{eqn_M_I_fun}; the stable equilibria are again as in Table~\ref{table_N_alpha_S_herding_Lethargic}.

\textit{ 
Thus even the individual consideration of environmental cost $e(z)$ was not effective 
%
in inducing a collective shift toward higher CT adoption. 
In some sense, individuals effectively act as bystanders with respect to altering environmental hazards (that negatively influence their own utilities), despite being active participants in the decision process.}



This negative result naturally raises the following question: what happens if $e(z)$ affects different groups in the population differently? In particular, can their decisions based on differentiated utilities influence the overall outcome?
 To examine this possibility,  we consider $n$  groups and  simply make $\rho$ in \eqref{utility_players} group dependent to obtain group-wise rational utilities:   

\vspace{-3mm}
{\small\begin{equation}\label{eqn_util_sens_ag}
 \hspace{-3mm}   u^E_{R,g_i}(a, z) = 
\begin{cases}
-P_c + (1 - z)z \, \m - \rho_i   e(z),  & \hspace{-1mm}\text{if } a = 1, \\[2pt]
-P_{uc} - (1 - z)z \, \m - \rho_i   e (z), & \hspace{-1mm}\text{if } a = 2.
\end{cases}
\end{equation}}%
However, clearly, the group-wise difference functions (defined as in \eqref{eqn_h_c_fun})  are exactly the same as in \eqref{eqn_h_fun}   for all the groups. Hence even the heterogeneity or extra-sensitive population could not drive towards a different outcome.

Thus  increasing the moral incentives or reducing the price disadvantage with CT or having a population with larger herding crowd  are the only factors that can improve CT adoption, while  the health hazards induced by the individual choices are completely ineffective. 

\hide{Before concluding we would like to remark that in \eqref{} and \eqref{},

is the only way the CT adoption can be improved  --- of course the existence of herding crowd improves the chances, but requires

This negative result naturally raises the following question: what happens if $e(z)$ affects different groups in the population differently? In particular, can their decisions influence the overall outcome?  To examine this possibility, we extend the analysis to populations composed of multiple groups with heterogeneous sensitivities to atmospheric $\coo$.

\subsection{Heterogeneous Sensitivity to $\coo$ Concentration}
We again consider a population composed of $R,H$ and $L$ type agents, but now the population is divided into $n$ groups $g_1, g_2, \dots, g_n$ that differ in their sensitivity to atmospheric $\coo$. Each group $g_i$ constitutes a fraction $f_i$ of the population, with $\sum_{i=1}^n f_i = 1$, and is characterized by a sensitivity parameter $r_i \geq 0$, which measures the perceived dis-utility from environmental damages.

Let $z_i$ denote the fraction of agents in group $g_i$ who adopt CT,  and let $z = \sum_{i=1}^n f_i z_i$ denote the aggregate CT adoption level in the population. Then the utility of a $R$ type agent belonging to group $g_i$ is defined as (see \eqref{eqn_utility_players1})
\begin{equation}\label{eqn_util_sens_ag}
 \hspace{-3mm}   u_{g_i}(a, z) = 
\begin{cases}
-P_1 + (1 - z)z \m - r_i \rho e(z),  & \hspace{-2mm}\text{if } a = 1, \\
-P_2 - (1 - z)z \m - r_i \rho e(z), & \hspace{-2mm}\text{if } a = 2,
\end{cases}
\end{equation}
where $\rho > 0$ and the term $r_i \rho e(z)$ captures the heterogeneous dis-utility induced by the steady-state $\coo$ concentration. The utilities of $H$ and $L$ type agents belonging to any group remain the same as in \eqref{eqn_util_herd_leth}. The corresponding utility difference function using \eqref{eqn_util_sens_ag} is
\begin{eqnarray}\label{eqn_h_c_fun_}
    h^s(z) :=  u(1, z) - u(2, z) = P_2 - P_1 + 2 (1-z) z \m,
\end{eqnarray}
which again coincides with \eqref{eqn_h_fun}. Hence, the set of stable MT-AMFEs remains the same as summarized in Table \ref{table_N_alpha_S_herding_Lethargic}.

This result shows that \textit{ heterogeneity in environmental exposure does not lead to different strategic incentives when environmental damages are independent of individual actions. Even groups that are more sensitive to $\coo$ face same incentive as others, since their individual choices cannot influence environmental outcomes. As a result, they cannot shift the population toward greater adoption of CTs. Thus, heterogeneity in environmental exposure alone does not change equilibrium outcomes.} 

Incentives that depend directly on actions -- such as subsidies, taxes, or moral rewards -- could potentially alter strategic behavior and increase adoption.

Before proceeding further, we note that, however in reality, individuals may misunderstand environmental impacts or underestimate their severity, that is, may not fully observe the $\coo$ concentration while decision making. If CT adoption remains limited even under full observation, this  indicates that the incentives to adopt CTs are weak. In more realistic situations where environmental damages are only partially (less) observed, the level of CT adoption would likely be even lower. Hence, analyzing the case with full observation remains useful.}

\section{Conclusion}
This paper studies clean technology (CT)  based products adoption in a large population  of consumers  with heterogeneous behavioral tendencies --- we consider   rational agents (trade-off moral incentives against price disadvantage of CT products),  herding crowd    (who follow the majority), and agents that exhibit inertia towards adopting new technologies.
  We identify and analyze   stable multi-type  mean-field equilibrium CT adoption levels (attractors of  a certain stochastic game dynamics) depending  upon  the price disadvantage, moral incentives, environmental ($\coo$) adversarial   effects, and the composition of the population. 

The realistic consideration of a variety of relevant behavioral tendencies,
along with some strategic   
dynamic decisions provides
several insights. When inertia is not too high in the population,  one can achieve widespread CT adoption, even with a big price disadvantage,  if the herding crowd constitutes a sufficient fraction -- influence through awareness campaigns can help. 
Morality incentives can be used to effectively   compel  the rational
crowd towards  CT, 
even if  the rest reject, when the latter proportion is not too high.  However, with a large proportion exhibiting inertia, moral pressure on rational agents can also break,  leading to  zero CT adoption.

Surprisingly, the inclusion of a negative predicted cost, proportional to the environmental damage resulting from continuing the  usage of  non-CT products,   did not alter the set of stable equilibria or the dynamic outcomes. 
Even the consideration of a highly sensitive rational population did not make a difference. In some sense, the rational agents   (in spite of actively
participating in  decision-making) become bystanders to their own environmental damage costs. 

\hide{

when a considerable
fraction of others choose the other-way.

When the rational crowd i  

When the number of agents with inertia is  small, widespread CT adoption requires either a sufficiently small price disadvantage 
or the presence of a substantial
herding crowd that can be influenced through awareness
campaigns. 
With intermediate fraction of agents with inertia, 
However, when the former agents are too many, they can also break the morality pressure of rational agents leading to  zero CT adoption.}

\hide{we find that individuals effectively behave as as bystanders with respect to environmental damage caused by $\coo$ concentration that negatively affect their own utilities, despite actively participating in the decision-making process. Furthermore, even heterogeneous sensitivities to environmental damages across groups do not change the equilibrium adoption levels. Future work could extend this framework by studying optimal policy design or incorporating richer network effects that influence technology adoption.}
 
\bibliographystyle{IEEEtran}
\vspace{-3mm}
\bibliography{Reference1}

@article{caetano2008optimal,
  author  = {Caetano, M. A. L. and Gherardi, D. F. M. and Yoneyama, T.},
  title   = {Optimal resource management control for {${CO}_2$} emission and reduction of the greenhouse effect},
  journal = {Ecological Modelling},
  volume  = {213},
  number  = {1},
  pages   = {119--126},
  year    = {2008}
}

@article{agarwal2025two,
  title={Two choice behavioral game dynamics with myopic-rational and herding players},
  author={Agarwal, Khushboo and Avrachenkov, Konstantin and Vyas, Raghupati and Kavitha, Veeraruna},
  journal={Proceedings of the ACM on Measurement and Analysis of Computing Systems},
  volume={9},
  number={1},
  pages={1--26},
  year={2025},
}

@article{carmona2022meanfield,
  author  = {R. Carmona and G. Dayanikli and M. Lauriere},
  title   = {Mean field models to regulate carbon emissions in electricity production},
  journal = {Dynamic Games and Applications},
  volume  = {12},
  number  = {3},
  pages   = {897--928},
  year    = {2022}
}

@book{filippov1960differential,
  title={Differential Equations with Discontinuous Right-Hand Side},
  author={Filippov, Aleksei F},
  journal={Matematicheskii sbornik},
  year={1988},
  publisher={Kluwer Academic Publishers}
}

@article{weitz2016oscillating,
  author  = {Weitz, Joshua S. and Eksin, Ceyhun and Paarporn, Keith and Brown, Sam P. and Ratcliff, William C.},
  title   = {Oscillating tragedy of the commons in replicator dynamics with environmental feedback},
  journal = {Proceedings of the National Academy of Sciences},
  volume  = {113},
  number  = {47},
  pages   = {E7518--E7525},
  year    = {2016},
  doi     = {10.1073/pnas.1604096113}
}

@article{tilman2020environmental,
  author  = {Tilman, Andrew R. and Plotkin, Joshua B.},
  title   = {Evolutionary games with environmental feedbacks},
  journal = {Nature Communications},
  volume  = {11},
  number  = {1},
  pages   = {915},
  year    = {2020},
  doi     = {10.1038/s41467-020-14531-6}
}

@book{nordhaus2008question,
  title={A question of balance: Weighing the options on global warming policies},
  author={Nordhaus, William},
  year={2008},
  publisher={Yale University Press}
}

@incollection{newell2010induced,
  title={The induced innovation hypothesis and energy-saving technological change},
  author={Newell, Richard G and Jaffe, Adam B and Stavins, Robert N},
  booktitle={Technological change and the environment},
  pages={97--126},
  year={2010},
  publisher={Routledge}
}

@article{joos2013carbon,
  author  = {Joos, F. and Roth, R. and Fuglestvedt, J. S. and Peters, G. P. et al.},
  title   = {Carbon dioxide and climate impulse response functions for the computation of greenhouse gas metrics: a multi-model analysis},
  journal = {Atmospheric Chemistry and Physics},
  volume  = {13},
  number  = {5},
  pages   = {2793--2825},
  year    = {2013}
}

@article{misra2013mathematical,
  title={A mathematical model to study the dynamics of carbon dioxide gas in the atmosphere},
  author={Misra, Arvind Kumar and Verma, Maitri},
  journal={Applied Mathematics and Computation},
  volume={219},
  number={16},
  pages={8595--8609},
  year={2013}
}

@article{vyas2026games,
  author  = {R. Vyas and K. Agarwal and K. Avrachenkov and V. Kavitha},
  title   = {Games with Rational and Herding Players},
  journal = {arXiv preprint arXiv:2602.02291},
  year    = {2026}
}

@book{sandholm2010population,
  author    = {W. H. Sandholm},
  title     = {Population Games and Evolutionary Dynamics},
  publisher = {MIT Press},
  address   = {Cambridge, MA},
  year      = {2010}
}

@article{aghion2016carbon,
  title={Carbon taxes, path dependency, and directed technical change: Evidence from the auto industry},
  author={Aghion, Philippe and Dechezlepr{\^e}tre, Antoine and Hemous, David and Martin, Ralf and Van Reenen, John},
  journal={Journal of Political Economy},
  volume={124},
  number={1},
  pages={1--51},
  year={2016},
  publisher={University of Chicago Press Chicago, IL}
}

@article{vyas2026multitype,
  author  = {R. Vyas and K. Das and V. Kavitha and S. Roy},
  title   = {Multi-type random game dynamics: limits at discontinuities and cyclic limits},
  journal = {arXiv preprint arXiv:2602.13032},
  year    = {2026}
}

@book{kushner2003stochastic,
  author    = {H. Kushner and G. Yin},
  title     = {Stochastic Approximation and Recursive Algorithms and Applications},
  publisher = {Springer},
  year      = {2003}
}

@book{borkar2009stochastic,
  title={Stochastic Approximation: A Dynamical Systems Viewpoint},
  author={Borkar, Vivek S},
  year={2009},
  publisher={Springer}
}

@inproceedings{agarwal2024balancing,
  title     = {Balancing rationality and social influence: Alpha-rational {Nash} equilibrium in games with herding},
  author    = {Agarwal, Khushboo and Avrachenkov, Konstantin and Kavitha, Veeraruna and Vyas, Raghupati},
  booktitle = {Proceedings of the International Conference on Game Theory for Networks (GameNets)},
  series    = {Lecture Notes in Computer Science},
  volume    = {13939},
  pages     = {91--107},
  year      = {2025},
  publisher = {Springer}
}

@article{acemoglu2012environment,
  author  = {D. Acemoglu and P. Aghion and L. Bursztyn and D. Hemous},
  title   = {The environment and directed technical change},
  journal = {American Economic Review},
  volume  = {102},
  number  = {1},
  pages   = {131--166},
  year    = {2012}
}
\appendix 
\noindent\textbf{Proof of Table  \ref{table_N_alpha_S_herding_Lethargic}:} By Definition \ref{defn_MT-MFE_binary} and using \eqref{eqn_utility_players1}, \eqref{eqn_util_herd_leth},\eqref{eqn_MT_MFE_BR}, an MT-AMFE  $z$~satisfies:
\begin{eqnarray}\label{eqnn_z_val}
    \hspace{-1mm} z = \alpha_R w + \alpha_H \mathbf{1}_{\{z \geq \nicefrac{1}{2}\}}, \   \support(w) \subseteq \arg\max_{a \in \{1,2\}} u_R(a,z),
\end{eqnarray}
where, $w:=\mu_R^*$,   denotes the fraction of rationals  choosing CT.
Define the left and right $\epsilon$–neighborhoods of $z$:
$$
N_\epsilon^-(z):=(z-\epsilon,z)\cap[0,1], \
N_\epsilon^+(z):=(z,z+\epsilon)\cap[0,1].
$$
\vspace{0.5mm}
\noindent\textbf{Case (i): when $\Delta_P \ge \nicefrac{\m}{2}$.}
Then for all $z\in[0,1]$, using~\eqref{eqn_utility_players1},
$$h(z):=u(1,z)-u(2,z)=2\m z(1-z)-\Delta_P \le \nicefrac{\m}{2}-\Delta_P \le 0.$$
Hence, for all $z\neq \nicefrac{1}{2}$, we have $h(z)<0$, implying $w=0$.
\begin{itemize}
\item if $z<\nicefrac{1}{2}$, then \eqref{eqnn_z_val} gives $z=\alpha_R w=0$.
\item if $z > \nicefrac12$, then \eqref{eqnn_z_val} gives $z=\alpha_H+\alpha_R w=\alpha_H$, which is valid iff $\alpha_H > \nicefrac{1}{2}$.
\item if $z = \nicefrac{1}{2}$ and $\Delta_P > \nicefrac{\m}{2}$, then again $h(z) < 0$ implies $w=0$. Hence by \eqref{eqnn_z_val}, we obtain $\nicefrac{1}{2} = \alpha_H$.
\item if $z = \nicefrac{1}{2}$ and $\Delta_P = \nicefrac{\m}{2}$, then $h(z) = 0$, implies any $w \in [0,1]$ is feasible. By \eqref{eqnn_z_val}, we have $\nicefrac{1}{2} = \alpha_H + \alpha_R w$ which provides a solution $w \in [0,1]$ iff $\alpha_H \leq \nicefrac{1}{2} \leq~\alpha_H+\alpha_R$.
\end{itemize}
Thus $z=0$ is always an equilibrium, and if $\alpha_H \geq \nicefrac{1}{2}$, then $z=\alpha_H$ is an equilibrium. If $\Delta_P = \nicefrac{\m}{2}$ then $z = \nicefrac{1}{2}$ is an equilibrium whenever $\alpha_H \leq \nicefrac{1}{2} \leq \alpha_H+\alpha_R$. 
 
We now check stability using Definition~\ref{defn_filippov}. From \eqref{eqn_M_I_fun}      there exists an $\epsilon>0$ such that $M^I(z)=-z<0$ for all $z\in N_\epsilon^+(0)$; hence $z=0$ is an attractor.
If $\alpha_H>\nicefrac12$, then there exists an $\epsilon>0$ such that $M^I(z)=\alpha_H-z>0$ for all $z\in N_\epsilon^-(\alpha_H)$ and $<0$ for all $z\in N_\epsilon^+(\alpha_H)$; hence $z=\alpha_H$ is an attractor.

Finally, when $\alpha_H=\nicefrac{1}{2}$, there exists $\epsilon>0$ such that $M^I(z)=-z<0$ for all $z\in N_\epsilon^-(\nicefrac{1}{2})$; hence $z=\nicefrac{1}{2}$ does not satisfy Definition~\ref{defn_filippov} and is not an attractor.

Thus the stable equilibria are $z^*_S=0$ always and $z^*_S=\alpha_H$ whenever $\alpha_H>\nicefrac{1}{2}$.

\noindent \textbf{Case (ii): when $\Delta_P<\nicefrac{\m}{2}$.} Then
$$
h(z)>0 \ \text{iff}\ z\in(R^-,R^+),
h(z)<0 \ \text{iff}\ z\in[0,R^-)\cup(R^+,1].
$$
\noindent If $z<\nicefrac12$ then from \eqref{eqnn_z_val}, we have $z=\alpha_R w$.
\begin{itemize}
\item if $z\in[0,R^-)$ then $h(z)<0$, so $w=0$ and hence $z=0$.
\item if $z\in(R^-,\nicefrac{1}{2})$ then $h(z)>0$, so $w=1$ and hence $z=\alpha_R$, which is consistent iff $R^-<\alpha_R<\nicefrac12$.
\item if $z=R^-$ then $h(z)=0$ and any $w\in[0,1]$ is possible; the equation $R^-=\alpha_R w$ is solvable with $w\in[0,1]$ iff $\alpha_R\ge R^-$.
\end{itemize}
Verifying as before using \eqref{eqn_M_I_fun}, $R^-$ is not stable, whereas the other equilibria (whenever they exist) are attractors.



\vspace{1mm}
\noindent if $z\ge \nicefrac12$ then from \eqref{eqnn_z_val}, we have $z=\alpha_R w + \alpha_H$.
\begin{itemize}
\item if $z\in[\nicefrac12,R^+)$ then $h(z)>0$, so $w=1$ and hence
$
z=\alpha_R+\alpha_H,
$ so this $z$ is an MFE  iff $\nicefrac{1}{2} \le \alpha_R+\alpha_H<R^+$.

\item if $z\in(R^+,1]$ then $h(z)<0$, so $w=0$ and hence $z=\alpha_H$.
This is consistent iff $\alpha_H>R^+$.

\item if $z=R^+$ then $h(z)=0$ and any $w\in[0,1]$ is possible; the equation
$
z=R^+=\alpha_H+\alpha_R w,
$
is solvable with $w\in[0,1]$ iff
$\alpha_H\le R^+\le \alpha_H+\alpha_R.$
\end{itemize}
Verifying as before using \eqref{eqn_M_I_fun}, among the equilibria   with  $z\ge \nicefrac12$,    $z=\alpha_R+\alpha_H$, $z=\alpha_H$, and $z=R^+$ (when they exist) are attractors, whereas $z=\nicefrac{1}{2}$ is not. \eop

\hide{
\subsection{Moral agents: always-clean technology}
\label{subsec_strongly_moral}
We first consider the case in which a fraction $\alpha_m$ of the population consists of strongly moral agents who always adopt the clean technology, regardless of prices or aggregate adoption levels. The remaining population is composed of rational and herding agents with proportions $\alpha_r$ and $\alpha_h$, respectively, satisfying $\alpha_r+\alpha_h+\alpha_m=1$.

Strongly moral agents provide a guaranteed lower bound on clean adoption and may shift the system into regions where rational agents favor clean technology. We extend
the notion of $\alpha$-Rational Nash Equilibrium to this three-type setting and characterize the resulting equilibrium adoption levels.

The following proposition describes the equilibrium adoption set in the presence of strongly moral agents.
\begin{prop}\label{prop:eq_three_types_cases} The equilibrium adoption set $\mathcal N_{\alpha_r}$ is given as:

\noindent when $\Delta_P>\nicefrac{\m}{2}$
$$
\mathcal N_{\alpha_r}=
\left\{
\begin{array}{ll}
\alpha_m, & \text{if } \alpha_m<\nicefrac{1}{2},\\
\alpha_m+\alpha_h, & \text{if } \alpha_r\le \nicefrac{1}{2}.
\end{array}
\right.
$$
\noindent when $\Delta_P=\nicefrac{\m}{2}$
$$
\mathcal N_{\alpha_r}=
\left\{
\begin{array}{ll}
\alpha_m, & \text{if } \alpha_m<\nicefrac{1}{2},\\
\alpha_m+\alpha_h, & \text{if } \alpha_r<\nicefrac{1}{2},\\
\nicefrac{1}{2}, & \text{if } \alpha_r\ge \nicefrac{1}{2}.
\end{array}
\right.
$$
\noindent when $\Delta_P<\nicefrac{\m}{2}$
$$
\mathcal N_{\alpha_r}=
\left\{
\begin{array}{ll}
\alpha_m, & \text{if } \alpha_m<\nicefrac{1}{2}\ \text{and}\ \alpha_m\le R^-,\\
\alpha_m+\alpha_r, & \text{if }  R^-<\alpha_m+\alpha_r < \nicefrac{1}{2},\\
R^-, & \text{if } \alpha_m\le R^-\le \alpha_m+\alpha_r,\\
\alpha_m+\alpha_h, & \text{if } \alpha_r\le R^-,\\
R^+, & \text{if } \alpha_r\ge R^-.
\end{array}
\right.
$$
Moreover, there is no equilibrium with 
$z\in\big[\nicefrac{1}{2},R^+\big).$
\end{prop}
\textbf{Proof:} The proof is provided in Appendix \ref{appendix}. \eop

We next analyze the long-run behavior under turn-by-turn dynamics and identify the stable equilibria that arise almost surely.
\begin{prop}\label{prop:eq_three_types_cases} The stable equilibrium adoption set $\mathcal N^S_{\alpha_r}$ is given as:

\noindent when $\Delta_P \geq \nicefrac{\m}{2}$
$$
\mathcal N^S_{\alpha_r}=
\left\{
\begin{array}{ll}
\alpha_m, & \text{if } \alpha_m<\nicefrac{1}{2},\\
\alpha_m+\alpha_h, & \text{if } \alpha_r < \nicefrac{1}{2}.
\end{array}
\right.
$$

\noindent when $\Delta_P<\nicefrac{\m}{2}$
$$
\mathcal N^S_{\alpha_r}=
\left\{
\begin{array}{ll}
\alpha_m, & \text{if }  \alpha_m < R^-,\\
\alpha_m+\alpha_r, & \text{if }\ R^-<\alpha_m+\alpha_r < \nicefrac{1}{2},\\
\alpha_m+\alpha_h, & \text{if } \alpha_r < R^-,\\
R^+, & \text{if } \alpha_r > R^-.
\end{array}
\right.
$$
Moreover, there is no equilibrium with 
$z\in\big[\nicefrac{1}{2},R^+\big).$
\end{prop}
These results show that strongly moral agents can promote clean adoption by anchoring a positive baseline level of adoption, which may activate rational incentives and reinforce herding behavior toward clean technology.

Since strongly moral agents always choose the clean action, implies strongly moral agents provide a guaranteed lower bound on clean adoption.

\smallskip
\noindent\textbf{When $\Delta_P\ge \nicefrac{\m}{2}$,}
rational players strictly prefer the unclean technology (except for indifference at $z=\nicefrac{1}{2}$ when $\Delta_P=\nicefrac{\m}{2}$), so moral and herding types are only supporting clean adoption. Therefore, herding cannot create a moderate level of adoption and only affects outcomes by coordinating the population toward either low adoption or high adoption.

\smallskip
\noindent\textbf{When $\Delta_P<\nicefrac{\m}{2}$,} rationals support clean adoption only when the aggregate adoption lies in the interval $(R^-,R^+)$. 
\begin{itemize}
    \item the moral-only equilibrium $z=\alpha_m$ persists whenever $\alpha_m\le R^-$, means that moral players alone may be insufficient to move the system into the region where rational agents favor clean adoption; and
    \item   if $\alpha_m+\alpha_h \ge R^+$, herding can sustain a high-adoption outcome $z=\alpha_m+\alpha_h$, even though rational players prefer the unclean technology at that level. 
    \item if $\alpha_m+\alpha_h \le R^+$, the system instead settles at the threshold $z=R^+$, where rational players are indifferent and only partially adopt the clean technology.
    
    \item if the combined share of moral and rational players, $\alpha_m+\alpha_r$, lies between $R^-$ and $R^+$, another equilibrium $z=\alpha_m+\alpha_r$ appears, where rational players adopt clean but herding follows the unclean majority. Otherwise, the system settles at the lower threshold $z=R^-$ with partial adoption by rational players.
\end{itemize}
Finally, since no equilibrium exists in $[\nicefrac{1}{2},R^+)$, adoption levels just above one half cannot be sustained. As soon as clean adoption reaches 50\%, herding switches to clean. At the same time, in this region rational players also prefer clean. So suddenly everyone follows clean, and adoption jumps to a very high level.
\section{Appendix} \label{appendix}
\noindent \textbf{Proof of Proposition \ref{prop:eq_three_types_cases}: }\label{prof_prop:eq_three_types_cases}
Since $z(1-z)\le \nicefrac{1}{4}$ on $[0,1]$, implies
$2z(1-z)\m \le \nicefrac{\m}{2}.$

\smallskip
\noindent Best response of rational players:
\begin{itemize}
    \item If $h(z)>0$, rational players choose action $1$, hence $w=1$.
    \item If $h(z)<0$, rational players choose action $2$, hence $w=0$.
    \item If $h(z)=0$, then any $w\in[0,1]$.
\end{itemize}
\smallskip
\noindent\textbf{Case analysis:}

\noindent\emph{Case (i): If $\Delta_P>\nicefrac{\m}{2}$}
then for all $z$,
$h(z)\le -\Delta_P + \nicefrac{\m}{2}<0,$
so $w=0$.
\begin{itemize}
    \item If $z<\nicefrac{1}{2}$, we get $z=\alpha_m$ and must have $\alpha_m<\nicefrac{1}{2}$.
    \item If $z\ge \nicefrac{1}{2}$, we get $z=\alpha_m + \alpha_h$ and must have $\alpha_m+\alpha_h\ge \nicefrac{1}{2}$.
\end{itemize}
\medskip
\noindent\emph{Case (ii): If $\Delta_P=\nicefrac{\m}{2}$} then for all $z$,
$$
h(z)=-\nicefrac{\m}{2}+2z(1-z)\m=\m ( 2z(1-z)-\nicefrac{1}{2})\leq 0,
$$
and equality holds only at $z=\nicefrac{1}{2}$, since $z(1-z)$ attains its unique maximum $\nicefrac{1}{4}$ at $z=\nicefrac{1}{2}$.
Hence for $z \neq \nicefrac{1}{2}$ we have $h(z)<0$ and thus $w=0$.
The same reasoning as in Case (i) yields equilibria $z=\alpha_m$ (when $\alpha_m < \nicefrac{1}{2}$) and $z=\alpha_m+\alpha_h$ (when $\alpha_m+\alpha_h>\nicefrac{1}{2}$). 

At $z=\nicefrac{1}{2}$, we have $h(\nicefrac{1}{2})=0$, so any $w\in[0,1]$ can be possible; and
herding players choose action $1$. Thus
$$
\nicefrac{1}{2}=\alpha_m+\alpha_h+\alpha_r w,
$$
which has a solution $w\in[0,1]$ iff $\alpha_m+\alpha_h\le \nicefrac{1}{2}$. In that case,
$$
\alpha_r w=\nicefrac{1}{2}-(\alpha_m+\alpha_h) \in [0,1].
$$

\medskip
\noindent\emph{Case (iii): $\Delta_P<\nicefrac{\m}{2}$.} The equation $h(z)=0$ has two roots are $R^-$ and $R^+$ as stated. Since $h$ is a concave quadratic in $z$,
it is positive on $(R^-,R^+)$ and negative on $[0,R^-)\cup(R^+,1]$.

\smallskip
\noindent\emph{Low-adoption region $z<\nicefrac{1}{2}$.} Here $z=\alpha_m+\alpha_r w$.
\begin{itemize}
    \item If $z\in[0,R^-)$ then $h(z)<0$ so $w=0$ and hence $z=\alpha_m$, requires $\alpha_m\le R^-$ and $\alpha_m < \nicefrac{1}{2}$.
\item If $z\in(R^-,\nicefrac{1}{2})$ then $h(z)>0$ so $w=1$ and hence $z=\alpha_m+\alpha_r$, requires
$R^-<\alpha_m+\alpha_r<R^+$ and $\alpha_m+\alpha_r<\nicefrac{1}{2}$.
\item If $z=R^-$ then $h(z)=0$, so any $w\in[0,1]$ is possible; the equation $R^-=\alpha_m+\alpha_r w$ is solvable with $w\in[0,1]$ iff $\alpha_m\le R^-\le \alpha_m+\alpha_r$, gives $\alpha_r w=R^- - \alpha_m$.
\end{itemize}

\smallskip
\noindent\emph{High-adoption region $z \ge \nicefrac{1}{2}$.} Here $z=\alpha_m+\alpha_h+\alpha_r w$.
\begin{itemize}
    \item If $z\in[\nicefrac{1}{2},R^+)$ then $h(z)>0$, so $w=1$, and the fixed-point equation yields
$z=\alpha_m+\alpha_h+\alpha_r=1$, contradicting $z<R^+<1$. Hence no equilibrium exists in $[\nicefrac{1}{2},R^+)$.
\item If $z\in(R^+,1]$ then $h(z)<0$, so $w=0$ and hence $z=\alpha_m+\alpha_h$, requires $\alpha_m+\alpha_h\ge R^+$.
\item If $z=R^+$ then $h(z)=0$ and any $w\in[0,1]$ is possible; the equation $R^+=\alpha_m+\alpha_h+\alpha_r w$
is solvable with $w \in [0,1]$ iff $\alpha_m+\alpha_h\le R^+\le 1$, giving
$\alpha_r w = R^+-(\alpha_m+\alpha_h)$. \eop
\end{itemize} 
}

\hide{
\subsection{Comparison and implications}
\label{subsec_extreme_comparison}

We now study long-run evolution of atmospheric $\coo$ concentration induced by the equilibrium adoption behavior.

We first characterize the set of MT-AMFEs in this setting, denoted by
$\mathcal{N}_{SA}$.
\begin{prop}\label{prop:eq_three_typ_dem1}
Consider the MgCT game with SA agents then ${\bm z}^* \in \mathcal{N}_{SA}$ as below.

\noindent when  $\Delta_P>\nicefrac{\m}{2}$
$$
z^* =
\left\{
\begin{array}{ll}
0, & \text{always},\\
\alpha_H, & \text{if }  \alpha_r+\alpha_{SA} \leq \nicefrac{1}{2}.
\end{array}
\right.
$$
\noindent when $\Delta_P=\nicefrac{\m}{2}$
$$
z^* =
\left\{
\begin{array}{ll}
0, & \text{always},\\
\alpha_H, & \text{if } \alpha_R+\alpha_{SA} < \nicefrac{1}{2},\\
\nicefrac{1}{2}, & \text{if } \nicefrac{1}{2} \leq \alpha_R+\alpha_{SA}\ \text{and}\ \alpha_{SA}\le \nicefrac{1}{2}.
\end{array}
\right.
$$
\noindent  when $\Delta_P<\nicefrac{\m}{2}$
$$
z^*=
\left\{
\begin{array}{ll}
0, & \text{always},\\
\alpha_R, & \text{if } R^-<\alpha_R<\nicefrac{1}{2},\\
R^-, & \text{if } \alpha_R\ge R^-,\\
\alpha_R+\alpha_H, & \text{if } R^- < \alpha_{SA} \leq \nicefrac{1}{2},\\
\alpha_H, & \text{if } \alpha_{SA}+\alpha_R < R^-,\\
R^+, & \text{if } \alpha_{SA} \leq R^- \leq \alpha_{SA}+\alpha_R.
\end{array}
\right.
$$
In addition, $z^*=\nicefrac{1}{2}$ is an equilibrium iff $\alpha_{SA}=\nicefrac{1}{2}$.
\end{prop}
\textbf{Proof:} The proof is provided in Appendix \ref{appendix}. \eop

Having characterized the equilibrium set, we next investigate which of these equilibria are stable under turn-by-turn dynamics.

Recently, the multi-type mean-field framework and the associated turn-by-turn dynamics have been analyzed in \cite{vyas2026multitype}. In particular,  \cite[Theorem~2]{vyas2026multitype} establishes almost sure convergence to singleton internally chain transitive (ICT) sets (see \cite[Definition 4]{vyas2026multitype}), and \cite[Theorem~5]{vyas2026multitype} shows that such singleton limit ICT sets  correspond to equilibrium points of the multi-type mean-field game. However, the stability or attractor properties of these  equilibria are not explicitly examined therein.

Here, we use these results of \cite{vyas2026multitype} and further analyze the stability of the resulting equilibria, using similar technique as in the rational–herding setting. Moreover, since the construction of turn-by-turn dynamics for multiple types follows the same principles as in rational and herding case, we do not repeat the full technical development here and instead focus on the stability of the equilibria and their implications.

\begin{prop}\label{prop:eq_three_typ_dem} Consider the MgCT game with SA agents then the stable equilibrium adoption $z^{*}_S$ is given as:

\medskip
\noindent when $\Delta_P \geq \nicefrac{\m}{2}$
$$
z^{*}_S =
\left\{
\begin{array}{ll}
0, & \text{always},\\
\alpha_H, & \text{if } \alpha_R+\alpha_{SA} < \nicefrac{1}{2}.
\end{array}
\right.
$$

\noindent when $\Delta_P<\nicefrac{\m}{2}$
$$
z^{*}_S=
\left\{
\begin{array}{ll}
0, & \text{always},\\
\alpha_R, & \text{if } R^-<\alpha_R<\nicefrac{1}{2},\\
\alpha_R+\alpha_H, & \text{if } R^- < \alpha_{SA} < \nicefrac{1}{2},\\
\alpha_H, & \text{if } \alpha_{SA}+\alpha_R < R^-,\\
R^+, &\text{if } \alpha_{um}\le R^-<\alpha_{SA}+\alpha_R.
\end{array}
\right.
$$ 
\end{prop}}

\end{document}